\documentclass[final]{siamltex}
\usepackage{graphicx}  
\usepackage{psfrag}
\usepackage{amsmath}
\usepackage{amssymb}
\usepackage{color}
\usepackage{afterpage}
\newtheorem{remark}[theorem]{Remark}
\DeclareMathOperator{\erfc}{erfc}
\DeclareMathOperator{\sech}{sech}
\DeclareMathOperator{\cosech}{cosech}

\begin{document}

\title{Dirichlet-Neumann and Neumann-Neumann Waveform Relaxation
  Algorithms for Parabolic Problems}

\author{Martin J. Gander\footnote[1]{Department of Mathematics,
    University of Geneva, Geneva, Switzerland.}\and
    Felix Kwok$^*$
    \and Bankim C. Mandal$^*$
}

\maketitle

\begin{abstract}
  We present and analyze a waveform relaxation version of the
  Dirichlet-Neumann and Neumann-Neumann methods for parabolic
  problems. Like the Dirichlet-Neumann method for steady problems, the
  method is based on a non-overlapping spatial domain decomposition,
  and the iteration involves subdomain solves with Dirichlet boundary
  conditions followed by subdomain solves with Neumann boundary
  conditions. For the Neumann-Neumann method, one step of the method
  consists of solving the subdomain problems using Dirichlet interface
  conditions, followed by a correction step involving Neumann
  interface conditions. However, each subdomain problem is now in
  space and time, and the interface conditions are also
  time-dependent.  Using Laplace transforms, we show for the heat
  equation that when we consider finite time intervals, the
  Dirichlet-Neumann and Neumann-Neumann methods converge superlinearly
  for an optimal choice of the relaxation parameter, similar to the
  case of Schwarz waveform relaxation algorithms. The convergence rate
  depends on the size of the subdomains as well as the length of the
  time window. For any other choice of the relaxation parameter,
  convergence is only linear. We illustrate our results with numerical
  experiments.
\end{abstract}

\section{Introduction}

We introduce and analyze new types of Waveform Relaxation (WR) methods
based on the Dirichlet-Neumann and Neumann-Neumann algorithms for
steady problems. To solve time-dependent problems in parallel, one can
either discretize in time to obtain a sequence of steady problems,
which is called Rothe's method after the German analyst Erich Rothe, and
then one applies domain decomposition algorithms to solve the steady
problems at each time step in parallel. Or one can first discretize in
space, which is called the method of lines, and then apply WR to the
large system of ordinary differential equations (ODEs) obtained from
the spatial discretization. WR methods have their origin in the work
of Picard \cite{Picard} and Lindel\"of \cite{Lind} for the existence proof of
solutions of ODEs in the late 19th century. Lelarasmee, Ruehli and
Sangiovanni-Vincentelli \cite{LelRue} were the first to introduce WR as a
parallel method for the solution of ODEs. For WR methods applied to
ODEs, we have two classical convergence results: (i) linear
convergence on unbounded time intervals under some dissipation
assumptions on the splitting (\cite{Nevan1}, \cite{Nevan2}, \cite{JelPo} and \cite{MiekNev}); and (ii)
superlinear convergence for nonlinear systems (including linear ones)
on bounded time intervals, assuming a Lipschitz condition on the
splitting function (\cite{Nevan1}, \cite{Nevan2}, \cite{BelZen} and \cite{Bjor}). The main
computational advantage of the WR method next to parallelization is
that one can use different time discretizations for different
components of the system.

Domain decomposition methods for elliptic PDEs can be extended to time
dependent problems by using the same decomposition in space, but then
solving time dependent problems in the subdomains during the iterative
solution process. This leads to WR type methods, see the early
references \cite{Bjor} and \cite{JelPo}. The systematic extension of the
classical Schwarz method to time dependent problems was started
independently in \cite{GanStu,GilKel}; Gander and Stuart \cite{GanStu} showed linear
convergence of overlapping Schwarz WR iteration for the heat equation
on unbounded time intervals with a rate depending on the size of the
overlap; Giladi and Keller \cite{GilKel} proved superlinear convergence of
the Schwarz WR method with overlap on bounded time intervals for the
convection-diffusion equation. Like WR algorithms in general, the so
called Schwarz Waveform Relaxation algorithms (SWR) converge
relatively slowly, except if the time window size is very short. A
remedy is to use optimized transmission conditions, which leads to
much faster algorithms, see \cite{GH1,BGH} for parabolic problems, and
\cite{GHN,GH2} for hyperbolic problems, and this technique also led to
the so called optimized Schwarz methods for elliptic problems, for an
overview, see \cite{Gand}.

The Dirichlet-Neumann and Neumann-Neumann methods belong to the class
of substructuring methods for solving elliptic PDEs.  The
Dirichlet-Neumann algorithm was first considered by Bj{\o}rstad \&
Widlund \cite{BjWid} and further studied in \cite{BramPas},
\cite{MarQuar01} and \cite{MarQuar02}; the Neumann-Neumann algorithm
was introduced by Bourgat et al. \cite{BouGT}, see also \cite{RoTal}
and \cite{TalRoV}. The performance of these algorithms for elliptic
problems is now well understood, see for example the book
\cite{TosWid}.  However, no substructuring-type analogue of the WR
method has been proposed so far. In this paper, we propose the
Dirichlet-Neumann Waveform Relaxation (DNWR) and the Neumann-Neumann
Waveform Relaxation (NNWR) methods, which generalize the use of
substructuring methods to the case of time-dependent problems in a
natural way. We define and analyze these methods in the continuous
setting to ensure the understanding of the asymptotic behavior of the
methods in the case of fine grids.

We formulate the new algorithms for the following parabolic equation
on a bounded domain $\Omega\subset\mathbb{R}^{d}$,
$0<t<T$, $d=1,2,3$:
\begin{equation}\label{modelproblem}
\begin{array}{rcll}
  \displaystyle\frac{\partial u}{\partial t}&=&\nabla\cdot\left(\kappa(\boldsymbol{x},t)\nabla u\right)+f(\boldsymbol{x},t),&  
   \boldsymbol{x}\in\Omega,\ 0<t<T, \\
  u(\boldsymbol{x},0)&=&u_{0}(\boldsymbol{x}), & \boldsymbol{x}\in\Omega,\\
  u(\boldsymbol{x},t)&=&g(\boldsymbol{x},t), & \boldsymbol{x}\in\partial\Omega,\ 0<t<T,
\end{array}
\end{equation} 
where $\kappa(\boldsymbol{x},t)\geq \kappa >0$.  In Section
\ref{Section2}, we introduce the non-overlapping DNWR algorithm with
two subdomains for the model problem (\ref{modelproblem}), and we
present sharp convergence estimates for DNWR obtained for the special
case of the one dimensional heat equation,
$\kappa(\boldsymbol{x},t)=1$. In Section \ref{Section3} we present the
NNWR algorithm for multiple subdomains for the general problem
(\ref{modelproblem}), and present sharp convergence estimates again
for the one dimensional heat equation.  Our convergence analysis shows
that both the DNWR and NNWR algorithms converge superlinearly on
finite time intervals, $T<\infty$.  It is based on detailed, technical
kernel estimates, which we show in Section \ref{Section4}. Section
\ref{Section5} contains the proofs of our main convergence results for
both DNWR and NNWR.  We then show in Section \ref{Section6} how the
analysis of the NNWR can be generalized to higher spatial dimensions,
and prove that the convergence estimates do not change.  We finally
show numerical results in Section \ref{Section7}, which illustrate our
analysis. We also test the algorithms in configurations not covered by
our analysis, and still observe the same convergence behavior. 

\section{The Dirichlet-Neumann Waveform Relaxation algorithm}\label{Section2}

To define the Dirichlet-Neumann WR algorithm for the model problem
(\ref{modelproblem}) on the space-time domain $\Omega\times(0,T)$ with
Dirichlet data given on $\partial\Omega$, we assume that the spatial
domain $\Omega$ is partitioned into two non-overlapping subdomains
$\Omega_{1}$ and $\Omega_{2}$, as illustrated in Figure
\ref{FigDecomp}. 
\begin{figure}
  \centering
  \includegraphics[width=0.5\textwidth]{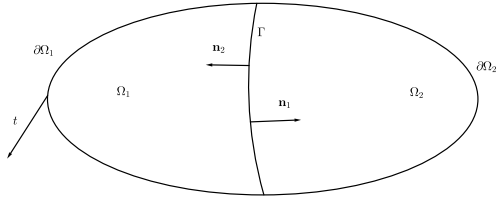}
  \vspace{-2.5em}
  \caption{Splitting into two non-overlapping subdomains}
  \label{FigDecomp}
\end{figure}
We denote by $u_{i}$ the restriction of the solution
$u$ of (\ref{modelproblem}) to $\Omega_{i}$, $i=1,2$, and by
$\boldsymbol{n}_{i}$ the unit outward normal for $\Omega_{i}$ on the
interface $\Gamma:=\partial\Omega_{1}\cap\partial\Omega_{2}$.
The Dirichlet-Neumann Waveform Relaxation algorithm consists of the
following steps: given an initial guess
$h^{0}(\boldsymbol{x},t)$ along the interface $\Gamma\times
(0,T)$, compute for $k=1,2,\ldots$ with 
$u_{1}^{k}=g$ on $\partial\Omega_{1}\setminus\Gamma$ and
$u_{2}^{k}=g$ on $\partial\Omega_{2}\setminus\Gamma$
the approximations
\begin{equation}\label{DNWR}
\arraycolsep0.08em
\begin{array}{rcll}
\partial_{t}u_{1}^{k}-\nabla\cdot\left(\kappa(\boldsymbol{x},t)\nabla u_{1}^{k}\right)&=&f, & \textrm{in}\; \Omega_{1},\\
u_{1}^{k}(\boldsymbol{x},0)&=&u_{0}(\boldsymbol{x}), & \textrm{in}\; \Omega_{1},\\
u_{1}^{k}&=&h^{k-1}, & \textrm{on}\; \Gamma,
\end{array}\ 
\begin{array}{rcll}
\partial_{t}u_{2}^{k}-\nabla\cdot\left(\kappa(\boldsymbol{x},t)\nabla u_{2}^{k}\right)&=&f, & \textrm{in}\; \Omega_{2},\\
u_{2}^{k}(\boldsymbol{x},0)&=&u_{0}(\boldsymbol{x}), & \textrm{in}\; \Omega_{2},\\
\partial_{\boldsymbol{n}_{2}} u_{2}^{k} & = & -\partial_{\boldsymbol{n}_{1}} u_{1}^{k}, & \textrm{on}\; \Gamma,
\end{array}
\end{equation}
and then update the value along the interface using
\begin{equation}\label{DNWR2}
  h^{k}(\boldsymbol{x},t)=\theta u_{2}^{k}\left|_{\Gamma\times(0,T)}\right.+(1-\theta)h^{k-1}(\boldsymbol{x},t),
\end{equation}
$\theta\in(0,1]$ being a relaxation parameter. Now the main goal of the
analysis is to study how the error
$h^{k-1}(\boldsymbol{x},t)-u\left|_{\Gamma\times(0,T)}\right.$
converges to zero, and by linearity it suffices to consider the so called
error equations, $f(\boldsymbol{x},t)=0$, $g(\boldsymbol{x},t)=0$,
$u_{0}(\boldsymbol{x})=0$ in (\ref{DNWR}), and examine how
$h^{k-1}(\boldsymbol{x},t)$ converges to zero as $k\rightarrow\infty$.

We present now sharp convergence estimates for algorithm
(\ref{DNWR},\ref{DNWR2}) for the special case of the heat equation,
$\kappa(\boldsymbol{x},t)=1$, on the one dimensional domain
$\Omega=(-a,b)$ with subdomains $\Omega_{1}=(-a,0)$ and
$\Omega_{2}=(0,b)$. Our convergence analysis is based on Laplace
transforms. We define the Laplace transform of a function $u(x,t)$
with respect to time $t$ as
\begin{equation}\label{LaplaceTransform}
  \hat{u}(x,s):=\mathcal{L}\left\{ u(x,t)\right\} 
    :=\int_{0}^{\infty}e^{-st}u(x,t)\, dt,
\end{equation}
where $s$ is a complex variable. If $\mathcal{L}\left\{ u(x,t)\right\}
=\hat{u}(x,s)$, then the inverse Laplace transform of $\hat{u}(x,s)$
is defined by
\begin{equation}\label{InverseLaplaceTransform}
  \mathcal{L}^{-1}\left\{ \hat{u}(x,s)\right\} :=u(x,t),\qquad t\geq0,
\end{equation}
and maps the Laplace transform of a function back to the original
function. For more information on Laplace transforms, see
\cite{Church,Oberhett}.

After a Laplace transform, the DNWR algorithm (\ref{DNWR},\ref{DNWR2})
for the error equations in the one dimensional heat equation setting
becomes
\begin{equation}\label{DNWRL}
\arraycolsep0.2em
\begin{array}{rcll}
  (s-\partial_{xx})\hat{u}_{1}^{k}&=&0 & \textrm{on $(-a,0)$},\\
  \hat{u}_{1}^{k}(-a,s)&=&0,\\
  \hat{u}_{1}^{k}(0,s)&=&\hat{h}^{k-1}(s),
\end{array}\quad
\begin{array}{rcll}
  (s-\partial_{xx})\hat{u}_{2}^{k}&=&0 & \textrm{on $(0,b)$},\\
  \partial_{x}\hat{u}_{2}^{k}(0,s)&=&\partial_{x}\hat{u}_{1}^{k}(0,s),\\
  \hat{u}_{2}^{k}(b,s)&=&0,
\end{array}
\end{equation}
followed by the updating step
\begin{equation}\label{DNWRL2}
  \hat{h}^{k}(s)=\theta\hat{u}_{2}^{k}(0,s)+(1-\theta)\hat{h}^{k-1}(s).
\end{equation}
Solving the two-point boundary value problems in the Dirichlet and
Neumann step in
(\ref{DNWRL}), we get
\begin{eqnarray}
  \hat{u}_{1}^{k}(x,s)&=&\frac{\hat{h}^{k-1}(s)}{\sinh(a\sqrt{s})}\sinh\left(
  (x+a)\sqrt{s}\right), \\
 \hat{u}_{2}^{k}(x,s)&=&\hat{h}^{k-1}(s)\frac{\coth(a\sqrt{s})}
   {\cosh(b\sqrt{s})}\sinh((x-b)\sqrt{s}).
\end{eqnarray}
By induction, we therefore find for the updating step the relation
\begin{equation}\label{DNWRLup}
  \hat{h}^{k}(s)=\left(
  1-\theta-\theta\tanh(b\sqrt{s})\coth(a\sqrt{s})\right)^{k}
  \hat{h}^{0}(s),\quad k=1,2,3,\ldots
\end{equation}

\begin{theorem}[Convergence of DNWR for $a=b$]
  When the subdomains are of the same size, $a=b$ in
  (\ref{DNWRL},\ref{DNWRL2}), the DNWR algorithm converges linearly for
  $0<\theta<1$, $\theta\neq1/2$. For $\theta=1/2$, it converges in two
  iterations. Convergence is independent of the time window size $T$.
\end{theorem}
\begin{proof}
For $a=b$, equation (\ref{DNWRLup}) reduces to
$\hat{h}^{k}(s)=(1-2\theta)^{k}\hat{h}^{0}(s),$ which has the simple
back transform $h^{k}(t)=(1-2\theta)^{k}h^{0}(t)$.  Thus the
convergence is linear for $0<\theta<1$, $\theta\neq1/2$. If $\theta=1/2$,
we have $h^{1}(t)=0$, and hence one more iteration produces the desired
solution on the entire domain.
\end{proof}

Having treated the simple case where the subdomains are of the same
size, $a=b$, we focus now on the more interesting case where $a\neq
b$. Defining
\begin{equation}\label{Gdef}
  G(s):=\tanh(b\sqrt{s})\coth(a\sqrt{s})-1
    =\frac{\sinh((b-a)\sqrt{s})}{\sinh(a\sqrt{s})\cosh(b\sqrt{s})},
\end{equation}
the recurrence relation (\ref{DNWRLup}) can be rewritten as
\begin{equation}\label{hrecurrenceLaplace}
\hat{h}^{k}(s)=\left\{\begin{array}{ll}
\left( q(\theta)-\theta G(s)\right)^{k}\hat{h}^{0}(s), & \theta\neq1/2,\\
\left(-1\right)^{k}2^{-k}G^{k}(s)\hat{h}^{0}(s), & \theta=1/2,
\end{array}\right.
\end{equation}
where $q(\theta)=1-2\theta$. Note that for $\textrm{Re}(s)>0,$ $G(s)$
is $\mathcal{O}(s^{-p})$ for every positive $p$, which can be seen as
follows: setting $s=re^{i\vartheta}$, we obtain for $a \geq b$
the bound
$$
  \left|s^{p}G(s)\right|\leq\left|\frac{s^{p}}{\cosh(b\sqrt{s})}\right|
  \leq\frac{2r^{p}}{\left|e^{b\sqrt{r/2}}-e^{-b\sqrt{r/2}}\right|}\rightarrow0
  \quad\mbox{as $r\rightarrow\infty$},
$$
and for $a<b$, we get the bound
$$
  \left|s^{p}G(s)\right|\leq\left|\frac{s^{p}}{\sinh(a\sqrt{s})}
  \right|\leq\frac{2r^{p}}{\left|e^{a\sqrt{r/2}}-e^{-a\sqrt{r/2}}\right|}\rightarrow0
  \quad\mbox{as $r\rightarrow\infty$}.
$$
Therefore, by \cite[p.~178]{Church}, $G(s)$ is the Laplace transform of an
infinitely differentiable function $F_{1}(t)$, which is the reason why we introduced
$G(s)$ in (\ref{Gdef}). We now define $F_{k}(t):=\mathcal{L}^{-1}\left\{
G^{k}(s)\right\}$ for $k=1,2,3,\ldots$. In what follows, we study the
special case $\theta=1/2$, when $h^{k}$ is given by a convolution of
$h^{0}$ with the analytic function $F_k$. For $\theta$ not equal to
$1/2$, different techniques are required to analyze the behavior of
the DNWR algorithm, and this will be done in a future paper. We also
have to consider two cases: $a>b$, which means that the Dirchlet
subdomain is bigger than Neumann subdomain, and $a<b$, when the
Neumann subdomain is bigger than the Dirichlet subdomain. We have
the following two convergence results, whose proofs will be given in 
Section \ref{Section5}.
\begin{theorem}[Convergence of DNWR for $a>b$]\label{Theorem2} 
  If $\theta=1/2$ and the Dirichlet subdomain is larger than the
  Neumann subdomain, then the error of the DNWR algorithm
  (\ref{DNWRL},\ref{DNWRL2}) satisfies for $t\in(0,\infty)$ the linear
  convergence estimate
  \begin{equation}
    \| h^{k}\|_{L^{\infty}(0,\infty)}\leq
    \left(\frac{a-b}{2a}\right)^{k}\| h^{0}\|_{L^{\infty}(0,\infty)}.
  \end{equation}
  On a finite time interval $t\in(0,T)$, the DNWR method converges
  superlinearly with the estimate
  \begin{equation}
    \| h^{k}\|_{L^{\infty}(0,T)}\leq
    \left(\frac{a-b}{a}\right)^{k}\erfc
    \left(\frac{kb}{2\sqrt{T}}\right)\| h^{0}\|_{L^{\infty}(0,T)}.
  \end{equation}
\end{theorem}
\begin{theorem}[Convergence of DNWR for $a<b$]\label{Theorem3} 
  If $\theta=1/2$ and the Dirichlet subdomain is smaller than the Neumann
  subdomain, then the error of the DNWR algorithm
  (\ref{DNWRL},\ref{DNWRL2}) satisfies for $t\in(0,\infty)$ the linear
  convergence estimate
  \begin{equation} \label{eq2.14}
    \| h^{2k}\|_{L^{\infty}(0,\infty)}\leq
    \left(\frac{b-a}{2a}\right)^{2k}\| h^{0}\|_{L^{\infty}(0,\infty)}.
  \end{equation}
  For a finite time interval $t\in (0,T)$, the DNWR converges superlinearly
  with the estimate
  \begin{equation}\label{eq2.15}
    \| h^{2k}\|_{L^{\infty}(0,T)}\leq
    \left( \frac{\sqrt{2}}{1-e^{-\frac{(2k+1)a^2}{T}}}\right)^{2k}
    e^{-k^{2}a^2/T}\| h^{0}\|_{L^{\infty}(0,T)}.
  \end{equation}
\end{theorem}

\begin{remark}
The linear estimate \eqref{eq2.14} does not always imply convergence,
because $b-a$ can be larger than $2a$. In other words, when $b > 3a$,
i.e., when the Neumann subdomain is much larger than the Dirichlet
one, the error over the infinite time interval does not converge to
zero as $k \to \infty$. In this case, one should switch the interface
conditions and solve a Dirichlet problem on the larger subdomain.
\end{remark}

\section{The Neumann-Neumann Waveform Relaxation algorithm}\label{Section3}

We now introduce the NNWR algorithm for the model problem
(\ref{modelproblem}) for multiple subdomains; for the case of two
subdomains in 1d, see \cite{Kwok}. Suppose $\Omega$ is partitioned
into non-overlapping subdomains $\{\Omega_{i}$, $1\leq i
\leq N\}$, as illustrated in Figure \ref{FigNeumannDecomp}. 
\begin{figure}
  \centering
  \includegraphics[width=0.5\textwidth]{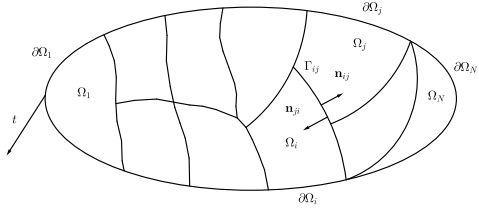}
 \vspace{-2.5em}
  \caption{Splitting into many non-overlapping subdomains}
  \label{FigNeumannDecomp}
\end{figure}
For $i=1,\ldots,N$ set
$\Gamma_i:=\partial\Omega_{i}\setminus\partial\Omega$,
$\Lambda_i:=\{j\in\{1,\ldots,N\}: \Gamma_{i}\cap\Gamma_{j} \, \mbox{has nonzero measure}\}$ and
$\Gamma_{ij}:=\partial\Omega_{i}\cap\partial\Omega_{j}$, so that the
interface of $\Omega_{i}$ can be rewritten as
$\Gamma_i=\bigcup_{j\in\Lambda_i}\Gamma_{ij}$. We denote by $\boldsymbol{n}_{ij}$
the unit outward normal for $\Omega_{i}$ on the interface
$\Gamma_{ij}$.
 
The NNWR algorithm starts with an initial guess
$w_{ij}^{0}(\boldsymbol{x},t)$ along the interfaces
$\Gamma_{ij}\times (0,T)$, $j\in\Lambda_i$, $i=1,\ldots,N$, and then
performs the following two-step iteration: at each iteration $k$, one
first solves Dirichlet problem on each $\Omega_{i}$ in parallel,
\begin{equation}\label{NNWRD}
  \begin{array}{rcll}
   \partial_{t}u_{i}^{k}-\nabla\cdot\left(\kappa(\boldsymbol{x},t)\nabla u_{i}^{k}\right)&=&f, & \mbox{in $\Omega_{i}$},\\
    u_{i}^{k}(\boldsymbol{x},0)&=&u_{0}(\boldsymbol{x}), & \mbox{in $\Omega_{i}$},\\
    u_{i}^{k}&=&g, & \mbox{on $\partial\Omega_{i}\setminus\Gamma_i$},\\
   u_{i}^{k}&=&w_{ij}^{k-1}, & \mbox{on $\Gamma_{ij}, j\in\Lambda_i$}.
  \end{array}
\end{equation}
One then solves Neumann problems on all subdomains,
\begin{equation}\label{NNWRN}
\begin{array}{rcll}
\partial_{t}\psi_{i}^{k}-\nabla\cdot\left(\kappa(\boldsymbol{x},t)\nabla \psi_{i}^{k}\right)&=&0, & \mbox{in $\Omega_{i}$},\\
   \psi_{i}^{k}(\boldsymbol{x},0)&=&0, & \mbox{in $\Omega_{i}$},\\
  \psi_{i}^{k}&=&0, & \mbox{on $\partial\Omega_{i}\setminus\Gamma_i$},\\
   \partial_{\boldsymbol{n}_{ij}}\psi_{i}^{k}&=&\partial_{\boldsymbol{n}_{ij}}u_{i}^{k}+\partial_{\boldsymbol{n}_{ji}}u_{j}^{k}, & \mbox{on $\Gamma_{ij}, j\in\Lambda_i$}.
\end{array}
\end{equation}
The interface values are then updated with the formula
\begin{equation}\label{NNWR2}
  w_{ij}^{k}(\boldsymbol{x},t)=w_{ij}^{k-1}(\boldsymbol{x},t)-\theta
  \left( \psi_{i}^{k}\left|_{\Gamma_{ij}\times(0,T)}\right.+\psi_{j}^{k}\left|_{\Gamma_{ij}\times(0,T)}\right.\right),
\end{equation}
where $\theta\in(0,1]$ is a relaxation parameter. 

As in the case of the DNWR algorithm, we prove our results first for
the one dimensional heat equation on the domain $\Omega:=(0,L)$ with
boundary conditions $u(0,t)=g_0(t)$ and $u(L,t)=g_L(t)$, for higher
spatial dimensions, see Section \ref{Section6}. We decompose $\Omega$
into non-overlapping subdomains $\Omega_{i}:=(x_{i-1},x_{i})$,
$i=1,\ldots,N$, and define the subdomain length
$h_{i}:=x_{i}-x_{i-1}$, and $h_{\min}:=\min_{1\leq i\leq N}h_{i}$. Our
initial guess is denoted by $\left\{w_{i}^{0}(t)\right\} _{i=1}^{N-1}$
on the interfaces $x_i$. By linearity, we again study the error
equations, $f=0$, $g_0=g_L=0$ and $u_0=0$, which leads with
$w_0^k(t)=w_N^k(t)=0$ for all $k$ to
\begin{equation}\label{NNWRL1}
  \begin{array}{rclrcl}
  \partial_t u_{i}^{k}-\partial_{xx} u_{i}^{k}&=&0, \qquad \textrm{in $\Omega_{i}$}, & 
    \partial_t\psi_{i}^k-\partial_{xx}\psi_{i}^k &=& 0,\qquad\textrm{in $\Omega_{i}$},\\
    u_{i}^{k}(x,0)&=&0, \qquad\textrm{in $\Omega_{i}$}, & \psi_{i}^{k}(x,0)&=&0, \qquad\textrm{in $\Omega_{i}$},\\
    u_{i}^k(x_{i-1},t) & = & w_{i-1}^{k-1}(t),&
    -\partial_{x}\psi_{i}^k(x_{i-1},t) & = & 
      (\partial_{x} u_{i-1}^k
      -\partial_{x} u_{i}^k)(x_{i-1},t),\\
   u_{i}^k(x_{i},t) & = & w_{i}^{k-1}(t),&
    \partial_{x}\psi_{i}^k(x_{i},t)&=&
     (\partial_{x} u_{i}^k
     -\partial_{x} u_{i+1}^k)(x_{i},t),
  \end{array}
\end{equation}
except for the first and last subdomain, where in the Neumann step the
Neumann conditions are replaced by homogeneous Dirichlet conditions
along the physical boundaries. The new
interface values for the next step are then defined as
\begin{equation}\label{NNWRL2}
  w_{i}^{k}(t)=w_{i}^{k-1}(t)
   -\theta\left(\psi_{i}^k(x_{i},t)+\psi_{i+1}^k(x_{i},t)\right) .
\end{equation}
We have the following convergence result for NNWR:
\begin{theorem}[Convergence of NNWR]\label{Theorem4}
  For $\theta=1/4$ and $T>0$ fixed, the NNWR algorithm (\ref{NNWRL1})--(\ref{NNWRL2}) converges superlinearly
  with the estimate
  \begin{equation}\label{Theorem4Result}
    \max_{1\leq i\leq N-1}\| w_{i}^{k}\|_{L^{\infty}(0,T)}
     \leq\left(\frac{\sqrt{6}}{1-e^{-\frac{(2k+1)h_{\min}^{2}}{T}}}\right)^{2k}
    e^{-k^{2}h_{\min}^{2}/T}\max_{1\leq i\leq N-1}
    \| w_{i}^{0}\|_{L^{\infty}(0,T)}.
  \end{equation}
\end{theorem}

The proof of Theorem \ref{Theorem4} will also be given in Section
\ref{Section5}. As in the case of the DNWR, the parameter choice
$\theta\ne 1/4$ requires different analysis techniques and is the subject
of further studies. %For $\theta=1/4$, a linear estimate which does
%not depend on $T$ has been given in {[}16{]}.

\section{Kernel estimates}\label{Section4}

The convergence results given in Theorems \ref{Theorem2},
\ref{Theorem3} and \ref{Theorem4} are based on technical estimates of
kernels arising in the Laplace transform of the DNWR and NNWR
algorithms. We present in this section the precise estimates needed.

\subsection{Properties of Laplace transforms}

We start with several elementary properties of positive functions and
their Laplace transforms. 
\begin{lemma}\label{SimpleLaplaceLemma}
Let $g$ and $w$ be two real-valued functions in $(0,\infty)$ with
$\hat{w}(s)=\mathcal{L}\left\{ w(t)\right\}$ the Laplace transform of
$w$. Then for $t\in(0,T)$, we have the following properties:
\begin{enumerate}
  \item \label{L1}If $g(t)\geq0$ and $w(t)\geq0$, the convolution
    $(g*w)(t):=\int_{0}^{t}g(t-\tau)w(\tau)d\tau\geq0$.

  \item \label{L2} $\| g*w\|_{L^{1}(0,T)}\leq\| g\|_{L^{1}(0,T)}\| w\|_{L^{1}(0,T)}.$

  \item \label{L3} $\bigl|(g*w)(t)\bigl|\leq\| g\|_{L^{\infty}(0,T)}\int_{0}^{T}\bigl|w(\tau)\bigl|d\tau.$

 \item \label{L4} $\int_{0}^{t}w(\tau)d\tau=(H*w)(t)=\mathcal{L}^{-1}\left(
   \frac{\hat{w}(s)}{s}\right) $, $H(t)$ being the Heaviside step
   function.
\end{enumerate}
\end{lemma}
\begin{proof}
  The proofs follow directly from the definitions.
\end{proof}
\begin{lemma}\label{LimitLemma}
  Let, $w(t)$ be a continuous and $L^{1}$-integrable function on
  $(0,\infty)$ with $w(t)\geq0$ for all $t\geq0$, and
  $\hat{w}(s)=\mathcal{L}\left\{ w(t)\right\}$ be its Laplace
  transform. Then, for $\tau>0$, we have the bound
  \begin{equation}
    \int_{0}^{\tau}|w(t)|dt\leq \lim_{s\rightarrow0+}\hat{w}(s).
  \end{equation}
\end{lemma}
\begin{proof}
With the definition of the Laplace transform (\ref{LaplaceTransform}) and
using positivity, we have
\begin{eqnarray*}
  \int_{0}^{\tau}|w(t)|dt & = & \int_{0}^{\tau}w(t)dt\leq\int_{0}^{\infty}w(t)dt
  =  \int_{0}^{\infty} \lim_{s\rightarrow0+}e^{-st}w(t)dt\\
  &= & \lim_{s\rightarrow0+}\int_{0}^{\infty}e^{-st}w(t)dt
 =  \lim_{s\rightarrow0+}\hat{w}(s),
\end{eqnarray*}
where the dominated convergence theorem was used to exchange the order
of limit and integration.
\end{proof}

\subsection{Positivity}
In order to use Lemma \ref{LimitLemma} in our analysis, we have to
show positivity of inverse transforms of kernels appearing in the DNWR
and NNWR iteration. These results are established in the following lemma.
\begin{lemma}\label{PositivityLemma}
  Let $\beta>\alpha\geq0$ and $s$ be a complex variable. Then, for
$t\in(0,\infty)$ 
\[
  \varphi(t):=\mathcal{L}^{-1}\left\{ \frac{\sinh(\alpha\sqrt{s})}
    {\sinh(\beta\sqrt{s})}\right\} \geq0\quad\mbox{and}\quad
  \psi(t):=\mathcal{L}^{-1}\left\{ \frac{\cosh(\alpha\sqrt{s})}
  {\cosh(\beta\sqrt{s})}\right\} \geq0.
\]
\end{lemma}
\begin{proof}
We first prove that $\varphi$ and $\psi$ are well-defined and continous functions
on $(0,\infty)$. Setting $s = re^{i\vartheta}$, a short calculation shows that for $\beta > \alpha \geq 0$ and for every positive $p$
$$ \left|\frac{s^p \sinh(\alpha\sqrt{s})}{\sinh(\beta\sqrt{s})}\right|
\leq r^p\cdot \left|\frac{e^{\alpha\sqrt{r/2}} + e^{-\alpha\sqrt{r/2}}}{e^{\beta\sqrt{r/2}} - e^{-\beta\sqrt{r/2}}}\right|
\to 0 \qquad \text{as $r \to \infty$,} $$
so by \cite[p.~178]{Church}, its inverse Laplace transform exists and is continuous (in fact, infinitely differentiable). Thus, $\varphi$ is a continuous function. A similar argument
holds for $\psi$.

Next, we prove the positivity of $\varphi$ and $\psi$ by noting that these kernels are related to
solutions of the heat equation. Let us consider the heat equation $u_{t}-u_{xx}=0$ on
$(0,\beta)$ with initial condition $u(x,0)=0$ and boundary conditions
$u(0,t)=0$, $u(\beta,t)=g(t)$. If $g$ is non-negative, then by the maximum principle,
this boundary value problem has a non-negative solution $u(\alpha,t)$ for all
$\alpha\in[0,\beta]$, $t>0$.  Now performing a Laplace transform of the
heat equation in time, we obtain the transformed solution along
$x=\alpha$ to be
\[
  \hat{u}(\alpha,s)
  =\hat{g}(s)\frac{\sinh(\alpha\sqrt{s})}{\sinh(\beta\sqrt{s})}
  \quad\Longrightarrow\quad 
  u(\alpha,t)=\int_{0}^{t}g(t-\tau)\varphi(\tau)d\tau.
\]
We now prove that $\varphi(t)\ge 0$ by contradiction: suppose
$\varphi(t_{0})<0$ for some $t_{0}>0$. Then by the continuity of
$\varphi$, there exists $\delta>0$ such that $\varphi(\tau)<0$ for
$\tau\in(t_{0}-\delta,t_{0}+\delta)$.  Now for $t>t_{0}+\delta,$ we
choose a non-negative $g$ as follows:
\[
g(\zeta)=\begin{cases}
1, & \zeta\in\left(t-t_{0}-\delta,t-t_{0}+\delta\right)\\
0, & \textrm{otherwise}.
\end{cases}
\]
Then
$u(\alpha,t)=\int_{t_{0}-\delta}^{t_{0}+\delta}g(t-\tau)\varphi(\tau)d\tau
=\int_{t_{0}-\delta}^{t_{0}+\delta}\varphi(\tau)d\tau<0$, which is a
contradiction, and hence $\varphi$ must be non-negative. To prove the
result for $\psi$, we use again the heat equation $u_{t}-u_{xx}=0$,
$u(x,0)=0$, but on the domain $(-b,b)$ and with boundary conditions
$u(-\beta,t)=u(\beta,t)=g(t)$. Using a Laplace transform in
time gives as solution at $x=\alpha$ 
$$
  \hat{u}(\alpha,s)=
   \hat{g}(s)\frac{\cosh(\alpha\sqrt{s})}{\cosh(\beta\sqrt{s})},
$$
and hence a similar argument as in the first case proves that $\psi$
is also non-negative.
\end{proof}

\subsection{Specific kernel estimates}
The following lemma contains specific estimates for the inverse
Laplace transform of two kernels in terms of infinite sums.
\begin{lemma}\label{KernelLemma}
For $k=1,2,3,\ldots$, we have the identities
\begin{eqnarray}\label{Kernel0}
  \mathcal{L}^{-1}\left(\cosech^{k}(\alpha\sqrt{s})\right)
    & = & 2^{k}{\displaystyle \sum_{m=0}^{\infty}}\binom{m+k-1}{m}
   \frac{(2m+k)\alpha}{\sqrt{4\pi t^3}}e^{-(2m+k)^2\alpha^2/4t},\\
   \label{Kernel2}
   \mathcal{L}^{-1}\left(\frac{\cosech^{k}(\alpha\sqrt{s})}{s}\right)
    & = & 2^{k}{\displaystyle \sum_{m=0}^{\infty}}\binom{m+k-1}{m}
   \erfc\left(\frac{(2m+k)\alpha}{2\sqrt{t}}\right).%,\\
%\label{Kernel1}
%  \mathcal{L}^{-1}\left(\frac{\sech^{k}(\beta\sqrt{s})}{s}\right)
%    & = & 2^{k} \sum_{m=0}^{\infty}(-1)^{m}\binom{m+k-1}{m}
%   \erfc\left(\frac{(2m+k)\beta}{2\sqrt{t}}\right).
\end{eqnarray}
In particular, %all three 
both functions are positive for $t > 0$.
\end{lemma}

\begin{proof}
Using that $\bigl|e^{-2\alpha\sqrt{s}}\bigl|<1$ for $\mbox{Re}(s)>0$,
we first expand $\mbox{cosech}$ into an infinite binomial series,
\begin{align}
  \textrm{cosech}^{k}(\alpha\sqrt{s}) & =\left(\frac{2}{e^{\alpha\sqrt{s}}-e^{-\alpha\sqrt{s}}}\right)^{k}=2^{k}e^{-k\alpha\sqrt{s}}\left(1-e^{-2\alpha\sqrt{s}}\right)^{-k}\nonumber \\ \label{eq4.4}
 & =2^{k}{\displaystyle \sum_{m=0}^{\infty}}\binom{m+k-1}{m}e^{-(2m+k)\alpha\sqrt{s}}.
\end{align}
Now using the inverse Laplace transform (see 
Oberhettinger \cite{Oberhett})
\begin{equation}\label{InversionFormula0}
  \mathcal{L}^{-1}\left( e^{-\lambda\sqrt{s}}\right) 
    =\frac{\lambda}{\sqrt{4\pi t^3}}e^{-\lambda^2/4t},
   \quad\lambda > 0,
\end{equation}
we obtain
\begin{eqnarray}
  \mathcal{L}^{-1}\left(\text{cosech}^{k}(\alpha\sqrt{s})\right)
  & = & 2^{k}\sum_{m=0}^{\infty}\binom{m+k-1}{m}
  \mathcal{L}^{-1}\left( e^{-(2m+k)\alpha\sqrt{s}}\right) \nonumber \\
  & = & 2^{k}\sum_{m=0}^{\infty}\binom{m+k-1}{m}\frac{(2m+k)\alpha}{\sqrt{4\pi t^3}}e^{-(2m+k)^2\alpha^2/4t}.\label{eq4.5a}
\end{eqnarray}
To justify taking the inverse Laplace transform term by term, we prove that the Laplace
transform of the right-hand side of \eqref{eq4.5a} indeed gives $\cosech^k(\alpha\sqrt{s})$.
Let $f_m(t) = 2^k{m+k-1\choose m}\frac{(2m+k)\alpha}{\sqrt{4\pi t^3}}e^{-(2m+k)^2\alpha^2/4t}$ be the $m$th
term of the series. Then for any real parameter $s_0 > 0$, we have for $\textrm{Re}(s) > s_0$
\begin{align*}
\int_0^\infty |e^{-st}f_m(t)| \,dt &\leq
2^k \int_0^\infty e^{-s_0t}{m+k-1\choose m}\frac{(2m+k)\alpha}{\sqrt{4\pi t^3}}e^{-(2m+k)^2\alpha^2/4t}\,dt\\
&= 2^k{m+k-1 \choose m}e^{-(2m+k)\alpha \sqrt{s_0}}.
\end{align*}
Thus, we have
$$ \sum_{m=0}^\infty \int_0^\infty |e^{-st}f_m(t)| \,dt
\leq \cosech^k(\alpha\sqrt{s_0}) < \infty. $$
This allows us to use Fubini's theorem, where the product measure is between the discrete counting measure and the Lebesgue measure on $[0,\infty)$. We thus obtain for all $\textrm{Re}(s) \geq s_0$
$$  \int_0^\infty e^{-st} \sum_{m=0}^\infty
f_m(t)\, dt = \sum_{m=0}^\infty \int_0^\infty e^{-st}f_m(t)\,dt = \cosech^k(\alpha\sqrt{s}). $$
The first identity \eqref{Kernel0} then follows by taking the inverse Laplace transform on both sides.
Since each $f_m(t)$ is positive for $t > 0$, we conclude that the limit function $\sum_{m=0}^\infty f_m(t)$ is also positive.

For the second identity, we need the inverse Laplace transform
\begin{equation}\label{InversionFormula}
  \mathcal{L}^{-1}\left( \frac{1}{s}e^{-\lambda\sqrt{s}}\right) 
    =\textrm{erfc}\left(\frac{\lambda}{2\sqrt{t}}\right),
   \quad\lambda > 0.
\end{equation}
By dividing the expansion \eqref{eq4.4} by $s$, we obtain
\begin{eqnarray}
  \mathcal{L}^{-1}\left(\frac{\text{cosech}^{k}(\alpha\sqrt{s})}{s}\right)
  & = & 2^{k}\sum_{m=0}^{\infty}\binom{m+k-1}{m}
  \mathcal{L}^{-1}\left( \frac{1}{s}e^{-(2m+k)\alpha\sqrt{s}}\right) \nonumber \\
  & = & 2^{k}\sum_{m=0}^{\infty}\binom{m+k-1}{m}\textrm{erfc}
  \left(\frac{(2m+k)\alpha}{2\sqrt{t}}\right),
\end{eqnarray}
where we justify the interchanging of sums and inverse transforms in the same way as above.
Since $\erfc$ is a positive function, so is the kernel \eqref{Kernel2}.
%
%Finally, for the third identity, we use the expansion
%\begin{equation}
%  \textrm{sech}^{k}(\beta\sqrt{s})=\left(\frac{2}
%    {e^{\beta\sqrt{s}}+e^{-\beta\sqrt{s}}}\right)^{k}=2^{k}
%    \sum_{m=0}^{\infty}(-1)^m\binom{m+k-1}{m}e^{-(2m+k)\beta\sqrt{s}},
%\end{equation}
%and with (\ref{InversionFormula}), we obtain 
%\begin{eqnarray*}
%  \mathcal{L}^{-1}\left(\frac{\textrm{sech}^{k}(\beta\sqrt{s})}{s}\right)
%  & = & 2^{k}\sum_{m=0}^{\infty}(-1)^m\binom{m+k-1}{m}\mathcal{L}^{-1}
%  \left( \frac{1}{s}e^{-(2m+k)\beta\sqrt{s}}\right) \\
%  & = & 2^{k} \sum_{m=0}^{\infty}(-1)^m\binom{m+k-1}{m}
%  \textrm{erfc}\left(\frac{(2m+k)\beta}{2\sqrt{t}}\right),
%\end{eqnarray*}
%which completes the proof for the third identity. By part 4 of Lemma \ref{SimpleLaplaceLemma}, this function is simply the integral of $\psi$ in Lemma \ref{PositivityLemma} from $0$ to $t$ when we set $\alpha = 0$, so it is positive.
\end{proof}

\section{Proofs of the main theorems}\label{Section5}

We now prove the main convergence results for the DNWR and NNWR
algorithms stated in Section \ref{Section2} and \ref{Section3}. 

\subsection{Proof of Theorem \ref{Theorem2} for DNWR} 

For $\theta=1/2$ we get from (\ref{hrecurrenceLaplace}) and using part
\ref{L3} of Lemma \ref{SimpleLaplaceLemma}
\begin{equation}\label{hkestimate}
  \left|h^{k}(t)\right|=\left|2^{-k}(-1)^{k}\left(h^{0}*F_{k}\right)(t)\right|
  \leq2^{-k}\| h^{0}\|_{L^{\infty}(0,T)}
  \int_{0}^{T}\left|F_{k}(\tau)\right|d\tau.
\end{equation}
So in order to get an $L^{\infty}$ convergence estimate, we need to
bound $\int_{0}^{T}\left|F_{k}(\tau)\right|d\tau$. Now in Theorem
\ref{Theorem2}, $a>b$, and $\mathcal{L}\left( -F_{1}(t)\right)
=\frac{\sinh((a-b)\sqrt{s})}{\sinh(a\sqrt{s})}\cdot\frac{1}{\cosh(b\sqrt{s})}$.
So by Lemma \ref{LimitLemma} and the fact that the convolution of two
positive functions is positive, see Lemma \ref{SimpleLaplaceLemma}
point \ref{L1}, $-F_{1}(t)$ is positive. Thus, $(-1)^{k}F_{k}(t)\geq0$
for all $t$, and we obtain from Lemma \ref{LimitLemma}
\[
  \int_{0}^{T}\left|(-1)^{k}F_{k}(\tau)\right|d\tau
  \leq \lim_{s\rightarrow0+}(-1)^{k}G^{k}(s)
  =\left(\frac{a-b}{a}\right)^{k}.
\]
This bound is valid for arbitrary values of $T$, and hence we get from
(\ref{hkestimate})
\[
  \| h^{k}\|_{L^{\infty}(0,\infty)}\leq
  \left(\frac{a-b}{2a}\right)^{k}\| h^{0}\|_{L^{\infty}(0,\infty)},
\]
which shows that the algorithm is converging at least linearly for
$a>b$. To get a more accurate bound, we define
$\hat{v}^{k}(s):=\cosh^{k}(b\sqrt{s})\hat{h}^{k}(s)$, and 
rewrite (\ref{hrecurrenceLaplace}) for $\theta=1/2$ as
\[
  \hat{v}^{k}(s)=2^{-k}\frac{\sinh^{k}((a-b)\sqrt{s})}
    {\sinh^{k}(a\sqrt{s})}\hat{v}^{0}(s).
\]
So if we set $g_{k}(t):=\mathcal{L}^{-1}\left(
\frac{\sinh^{k}((a-b)\sqrt{s})}{\sinh^{k}(a\sqrt{s})}\right)$, then
part \ref{L3} of Lemma \ref{SimpleLaplaceLemma} yields
\[
  \| v^{k}\|_{L^{\infty}(0,T)}\leq
  2^{-k}\| v^{0}\|_{L^{\infty}(0,T)}\int_{0}^{T}|g_{k}(\tau)|d\tau.
\]
By Lemma \ref{LimitLemma},
$\int_{0}^{T}g_{k}(\tau)d\tau\leq\left(\frac{a-b}{a}\right)^{k}$, and
we therefore obtain
\begin{equation}\label{5.2}
  \| v^{k}\|_{L^{\infty}(0,T)}\leq
  \left(\frac{a-b}{2a}\right)^{k}\| v^{0}\|_{L^{\infty}(0,T)}.
\end{equation}
Setting $f_{k}(t):=\mathcal{L}^{-1}\left(\frac{1}{\cosh^{k}(b\sqrt{s})}\right)$
we have
\[
  h^{k}(t)=\left(f_{k}*v^{k}\right)(t)=\int_{0}^{t}f_{k}(t-\tau)v^{k}(\tau)d\tau,
\]
from which it follows, using again part \ref{L3} of Lemma
\ref{SimpleLaplaceLemma} that
\begin{equation}\label{5.3}
  \| h^{k}\|_{L^{\infty}(0,T)}\leq
  \| v^{k}\|_{L^{\infty}(0,T)}\int_{0}^{T}|f_{k}(\tau)|d\tau.
\end{equation}
By Lemma \ref{PositivityLemma}, $f_{k}(t)\geq0$ for all $t$. To obtain a bound for
$\int_0^T f_k(\tau)\,d\tau$, we first show that the function
$r_k(t) = \mathcal{L}^{-1}(2^k e^{-kb\sqrt{s}})$ is greater than or equal to $f_k(t)$ for all $t > 0$,
and then bound $\int_0^T r_k(\tau)\,d\tau$ instead. Indeed, we have
\begin{align*}
\mathcal{L}\left\{r_k(t) - f_k(t)\right\} &=
2^ke^{-kb\sqrt{s}} - \frac{2^k}{(e^{b\sqrt{s}} + e^{-b\sqrt{s}})^k}\\
&= \frac{2^k((1 + e^{-2b\sqrt{s}})^k - 1)}{(e^{b\sqrt{s}}+e^{-b\sqrt{s}})^k}\\
&= \sum_{j=1}^k {k\choose j}e^{-2jb\sqrt{s}}\sech^k(b\sqrt{s}).
\end{align*}
Note that in addition to $f_k(t) = \mathcal{L}^{-1}(\sech^k(b\sqrt{s}))$, $\mathcal{L}^{-1}(e^{-2jb\sqrt{s}})$ is also a positive function for $j=1,\ldots, k$; see \eqref{InversionFormula0}. Thus, $\mathcal{L}^{-1}(e^{-2jb\sqrt{s}}\sech^k(b\sqrt{s}))$ is a convolution of positive functions, and hence positive by part 1 of Lemma \ref{SimpleLaplaceLemma}. This implies $r_k(t) - f_k(t) \geq 0$, so we deduce that
%Thus by
%part \ref{L4} of Lemma \ref{SimpleLaplaceLemma}, and using the kernel
%estimate (\ref{Kernel1}) from Lemma \ref{KernelLemma}, we get
%\begin{equation}
%  \arraycolsep0.2em
%  \begin{array}{rcl}
%    \displaystyle 
%    \int_{0}^{T}f_{k}(\tau)d\tau 
%    &=&\displaystyle
%     \mathcal{L}^{-1}\left(\frac{\text{sech}^{k}(b\sqrt{s})}{s}\right)_{t=T}\\
%    &=&\displaystyle2^{k}\sum_{m=0}^{\infty}(-1)^{m}\binom{m+k-1}{m}
%    \textrm{erfc}\left(\frac{(2m+k)b}{2\sqrt{T}}\right)\\
%  \displaystyle
%    &=&\displaystyle2^{k}\left(\textrm{erfc}\left(\frac{kb}{2\sqrt{T}}\right)
%    -\sum_{m=0}^{\infty}\left(\binom{k+2m}{2m+1}\textrm{erfc}
%    \left(\frac{(4m+2+k)b}{2\sqrt{T}}\right)\right.\right.\\
%    & & \hspace{11em}\displaystyle\left.\left.
%    -\binom{k+2m+1}{2m+2}\textrm{erfc}
%    \left(\frac{(4m+4+k)b}{2\sqrt{T}}\right)\right) \right)\\
%    &\leq& \displaystyle 2^{k}\textrm{erfc}\left(\frac{kb}{2\sqrt{T}}\right),
%  \end{array}
%\end{equation}
%where we used that each term inside the infinite sum is non-negative.
%We therefore obtain the bound
\[
  \int_{0}^{T}f_{k}(\tau)\,d\tau \leq \int_0^T r_k(\tau)\,d \tau
= \mathcal{L}^{-1}\left(\frac{2^{k}e^{-kb\sqrt{s}}}{s}\right) = 2^k\, \textrm{erfc}
  \left(\frac{kb}{2\sqrt{T}}\right),
\]
where we expressed the second integral as an inverse Laplace transform using
Lemma \ref{SimpleLaplaceLemma}, part \ref{L4}, which we then evaluated using
\eqref{InversionFormula}.
Finally, we combine the bound above with (\ref{5.2}) and (\ref{5.3}) to obtain the
second estimate of Theorem \ref{Theorem2}, which concludes the
proof of this theorem.

\subsection{Proof of Theorem \ref{Theorem3} for DNWR}

For the case $a<b$, a linear estimate which does not depend on $T$ has been given in \cite{Mandal}. For the second part, we rewrite (\ref{hrecurrenceLaplace}) in the
form
\begin{equation}\label{hrec2}
  \hat{h}^{2k}(s)=\left(-\frac{1}{2}\right)^2G(s)\hat{h}^{2k-1}(s)
    =\frac{1}{4}G^{2}(s)\hat{h}^{2k-2}(s),
\end{equation}
and defining
$\hat{\phi}(s):=\frac{\sinh^{2}((b-a)\sqrt{s})}{\cosh^{2}(b\sqrt{s})}$,
$\hat{h}_{1}(s)=\frac{\cosh^{2}((b-a)\sqrt{s})}{\cosh^{2}(b\sqrt{s})}$
and $\hat{h}_{2}(s)= \frac{1}{\cosh^{2}(b\sqrt{s})}$, we can rewrite
\[
  G^{2}(s) = \frac{\sinh^{2}((b-a)\sqrt{s})}
   {\sinh^{2}(a\sqrt{s})\cosh^{2}(b\sqrt{s})}
  =\frac{1}{\sinh^{2}(a\sqrt{s})}\hat{\phi}(s)
  =\frac{1}{\sinh^{2}(a\sqrt{s})}\left(\hat{h}_1(s)-\hat{h}_2(s)\right).
\]
This motivates the definition of the new sequence
$\hat{\vartheta}^{2k}(s):=\sinh^{2k}(a\sqrt{s})\hat{h}^{2k}(s)$, which
from (\ref{hrec2}) satisfies the recurrence
\[
  \hat{\vartheta}^{2k}(s)=\frac{1}{4}\hat{\phi}(s)\hat{\vartheta}^{2k-2}(s).
\]
Now using part \ref{L3} of Lemma \ref{SimpleLaplaceLemma}, we obtain
the estimate
\begin{equation}
  \|\vartheta^{2k}\|_{L^{\infty}(0,T)}
  \leq\frac{1}{4}\|\vartheta^{2k-2}
  \|_{L^{\infty}(0,T)}\int_{0}^{T}|\phi(\tau)|d\tau,
\end{equation}
and using Lemma \ref{LimitLemma} leads to
\begin{eqnarray*}
  \int_{0}^{T}|\phi(\tau)|d\tau & \leq & 
  \int_{0}^{T}h_{1}(\tau)d\tau+\int_{0}^{T}h_{2}(\tau)d\tau\\
  &\leq& \lim_{s\rightarrow0+}
    \frac{\cosh^{2}((b-a)\sqrt{s})}{\cosh^{2}(b\sqrt{s})}
    +\lim_{s\rightarrow0+}\frac{1}{\cosh^{2}(b\sqrt{s})}=2.
\end{eqnarray*}
By induction, we therefore obtain
\begin{equation}\label{5.6}
  \|\vartheta^{2k}\|_{L^{\infty}(0,T)}
  \leq\frac{1}{2^{k}}\|\vartheta^{0}\|_{L^{\infty}(0,T)}
   =\frac{1}{2^{k}}\| h^{0}\|_{L^{\infty}(0,T)}.
\end{equation}
Now defining
$\varphi_{2k}(t)=\mathcal{L}^{-1}\left(\frac{1}{\sinh^{2k}(a\sqrt{s})}\right)$,
we have
\[
  h^{2k}(t)=\left(\varphi_{2k}*\vartheta^{2k}\right)(t)
  =\int_{0}^{t}\varphi_{2k}(t-\tau)\vartheta^{2k}(\tau)d\tau,
\]
from which it follows by part \ref{L3} of Lemma
\ref{SimpleLaplaceLemma} and using (\ref{5.6}) that
\begin{equation}\label{5.7}
  \| h^{2k}\|_{L^{\infty}(0,T)}\leq 
  \|\vartheta^{2k}\|_{L^{\infty}(0,T)}
  \int_{0}^{T}|\varphi_{2k}(\tau)|d\tau
  \leq B(k,T)\| h^{0}\|_{L^{\infty}(0,T)},
\end{equation}
where $B(k,T):=\frac{1}{2^{k}}\int_{0}^{T}|\varphi_{2k}(\tau)|d\tau$.
By Lemma \ref{KernelLemma}, $\varphi_{2k}(t)\geq0$ for all $t>0$.
Thus by part \ref{L4} of Lemma \ref{SimpleLaplaceLemma}, and then
using equation (\ref{Kernel2}) from Lemma \ref{KernelLemma}, we get
\begin{eqnarray*}
  B(k,T) & = & \frac{1}{2^{k}}\mathcal{L}^{-1}\left(
   \frac{\textrm{cosech}^{2k}(a\sqrt{s})}{s}\right)_{t=T}\\
   &=&2^{k}\sum_{m=0}^{\infty}\binom{m+2k-1}{m}\textrm{erfc}
   \left(\frac{(m+k)a}{\sqrt{T}}\right)\\
   & \leq & 2^{k}\sum_{m=0}^{\infty}\binom{m+2k-1}{m}
     \exp\left(-\frac{(m+k)^{2}a^{2}}{T}\right)\\
   & = & 2^{k}e^{-k^{2}a^2/T}\sum_{m=0}^{\infty}\binom{m+2k-1}{m}
     \exp\left(-\frac{(m^2+2km)a^{2}}{T}\right)\\
   & \leq & 2^{k}e^{-k^{2}a^2/T}\sum_{m=0}^{\infty}\binom{m+2k-1}{m}
     \exp\left(-\frac{m(2k+1)a^{2}}{T}\right)\\
   & \leq & \left(\frac{\sqrt{2}}{1-e^{-\frac{2k+1}{\sigma}}}\right)^{2k}
    e^{-k^{2}/\sigma},\quad \textrm{with $\sigma:=T/a^{2}$},
\end{eqnarray*}
where we used for the first inequality the estimate
\[
  \textrm{erfc}(x)=\frac{2}{\sqrt{\pi}}\int_{x}^{\infty}e^{-t^{2}}dt
   =\frac{2}{\sqrt{\pi}}\int_{0}^{\infty}e^{-(x+t)^{2}}dt
   \leq\frac{2}{\sqrt{\pi}}e^{-x^{2}}\int_{0}^{\infty}e^{-t^{2}}dt
   \leq e^{-x^{2}},
\]
and for the third inequality
\[
   \frac{1}{(1-z)^{\beta}}=\sum_{m\geq0}\binom{m+\beta-1}{m}z^{m},
   \quad\mbox{for $|z|<1$}.
\]
Inserting the estimate for $B(k,T)$ into (\ref{5.7}) gives then the
result of the theorem.

\begin{remark}
Note that the factor multiplying $e^{-k^2a^2/T}$ in the estimate \eqref{eq2.15} is an increasing function of $k$ in general, since
$\frac{\sqrt{2}}{1-e^{-\frac{2k+1}{\sigma}}}>1$. Thus, the bound \eqref{eq2.15}
may increase initially for small iteration numbers $k$, before the factor $e^{-k^{2}a^2/T}$ starts dominating and causing the bound to decrease to zero superlinearly. To estimate the turning point, let us fix an integer $l > 0$ and consider the behavior of the algorithm for iteration numbers $k > 2l$. Then by writing $\alpha = e^{-l/\sigma}$, we can bound $B(k,T)$ by
\begin{align*}
  B(k,T)&=\left(\frac{\sqrt{2}}
   {1-e^{-\frac{2k+1}{\sigma}}}\right)^{2k}e^{-2kl/\sigma} e^{-k(k-2l)/\sigma}\\
   &\leq \left(\frac{\sqrt{2}e^{-l/\sigma}}
   {1-e^{-\frac{2k}{\sigma}}}\right)^{2k}e^{-k(k-2l)/\sigma}
   \leq {\underbrace{\left(\frac{\sqrt{2}\alpha}{1-\alpha^{4}}
   \right)}_{=(*)}}^{2k}e^{-\frac{(k-2l)^{2}}{\sigma}}.
\end{align*}
Thus, if $\sqrt{2}\alpha/(1-\alpha^4) < 1$, then the factor ($*$) is less than 1 and the bound $B(k,T)$ contracts superlinearly for $k > 2l$. This is true whenever $\alpha < \alpha_0$, where $\alpha_0 \approx 0.6095$ is the unique positive root of $\psi(\alpha)=\alpha^{4}+\sqrt{2}\alpha-1$.  
Hence, we get superlinear convergence for $k>2l>0.99T/a^{2}.$ 

\end{remark}

%% \begin{remark}
%% We may look deeper for an intuitive idea to get superlinear convergence
%% for small time windows. Figure 5.1 shows $F_{k}(t)$ for $k=1,2,3$;
%% we see that the curves shift to the right and at the same time the
%% peak decreases as $k$ increases. So, if one only considers a small
%% time window, the peak will eventually exit the time window for $k$
%% large enough and its contribution will be vanishingly small in the
%% expression (2.4).
%% \begin{figure}[H]
%% \centering{}\includegraphics[width=9cm,height=5cm]{convol}\caption{$F_{k}(t),k=1,2,3$.}
%% \end{figure}

%% \end{remark}

\subsection{Proof of Theorem \ref{Theorem4} for NNWR}

We start by applying the Laplace transform to the homogeneous 
Dirichlet subproblems in (\ref{NNWRL1}), and obtain  
\[
  s\hat{u}_{i}-\hat{u}_{i,xx}=0,\quad 
  \hat{u}_{i}(x_{i-1},s)=\hat{w}_{i-1}(s),\quad
  \hat{u}_{i}(x_{i},s)=\hat{w}_{i}(s),
\]
for $i=2,...,N-1.$ These subdomain problems have the solutions
\[
  \hat{u}_{i}(x,s)=\frac{1}{\sinh(h_{i}\sqrt{s})}
   \left(\hat{w}_{i}(s)\sinh\left((x-x_{i-1})\sqrt{s}\right)
         +\hat{w}_{i-1}(s)\sinh\left((x_{i}-x)\sqrt{s}\right)\right).
\]
Next we apply the Laplace transform to the Neumann subproblems
(\ref{NNWRN}) for subdomains not touching the physical boundary, and
obtain
\[
  \hat{\psi}_{i}(x,s)=C_{i}(s)\cosh\left((x-x_{i-1})\sqrt{s}\right)
    +D_{i}(s)\cosh\left((x_{i}-x)\sqrt{s}\right),
\]
where the notation $\sigma_{i}:=\sinh\left(h_{i}\sqrt{s}\right)$
and $\gamma_{i}:=\cosh\left(h_{i}\sqrt{s}\right)$ gives
\begin{eqnarray*}
  C_{i} & = & \frac{1}{\sigma_{i}}\left(\hat{w}_{i}
  \left(\frac{\gamma_{i}}{\sigma_{i}}+\frac{\gamma_{i+1}}{\sigma_{i+1}}\right)
  -\frac{\hat{w}_{i-1}}{\sigma_{i}}-\frac{\hat{w}_{i+1}}{\sigma_{i+1}}\right),\\
  D_{i} & = & \frac{1}{\sigma_{i}}\left(\hat{w}_{i-1}
  \left(\frac{\gamma_{i}}{\sigma_{i}}+\frac{\gamma_{i-1}}{\sigma_{i-1}}\right)
  -\frac{\hat{w}_{i-2}}{\sigma_{i-1}}-\frac{\hat{w}_{i}}{\sigma_{i}}\right).
\end{eqnarray*}
We therefore obtain for $i=2,...,N-2,$ at iteration $k$
\begin{eqnarray*}
  \hat{w}_{i}^{k}(s) & = & \hat{w}_{i}^{k-1}(s)-\theta
    \left(\hat{\psi}_{i}^{k}(x_{i},s)+\hat{\psi}_{i+1}^{k}(x_{i},s)\right)\\
  &=& \hat{w}_{i}^{k-1}(s)-\theta\left( 
    C_{i}\gamma_{i}+D_{i}+C_{i+1}+D_{i+1}\gamma_{i+1}\right).
\end{eqnarray*}
Using the identity $\gamma_{i}^{2}-1=\sigma_{i}^{2}$ and simplifying,
we get
\begin{equation}\label{5.9}
\begin{array}{rcl}
  \hat{w}_{i}^{k}&=&\displaystyle\hat{w}_{i}^{k-1}-\theta\left(\hat{w}_{i}^{k-1}
    \left(2+\frac{2\gamma_{i}\gamma_{i+1}}{\sigma_{i}\sigma_{i+1}}\right)
    +\frac{\hat{w}_{i+1}^{k-1}}{\sigma_{i+1}}
    \left(\frac{\gamma_{i+2}}{\sigma_{i+2}}-\frac{\gamma_{i}}{\sigma_{i}}\right)
   \right.\\
  &&\hspace{4em}\displaystyle\left.+\frac{\hat{w}_{i-1}^{k-1}}{\sigma_{i}}
 \left(\frac{\gamma_{i-1}}{\sigma_{i-1}}-\frac{\gamma_{i+1}}{\sigma_{i+1}}\right)
  -\frac{\hat{w}_{i+2}^{k-1}}{\sigma_{i+1}\sigma_{i+2}}
  -\frac{\hat{w}_{i-2}^{k-1}}{\sigma_{i}\sigma_{i-1}}\right).
\end{array}
\end{equation}
For $i=1$ and $i=N$, the Neumann conditions on the physical
boundary are replaced by homogeneous Dirichlet conditions $\psi_{1}(0,t)=0$ and $\psi_{N}(L,t)=0$, $t>0$. For these two subdomains, we obtain 
as solution after a Laplace transform
\begin{eqnarray*}
  \hat{\psi}_{1}(x,s)&=&
    \frac{1}{\gamma_{1}}\left(\hat{w}_{1}\left(
      \frac{\gamma_{1}}{\sigma_{1}}+\frac{\gamma_{2}}{\sigma_{2}}\right)
      -\frac{\hat{w}_{2}}{\sigma_{2}}\right)
      \sinh\left((x-x_{0})\sqrt{s}\right),\\
  \hat{\psi}_{N}(x,s)&=&\frac{1}{\gamma_{N}}\left(\hat{w}_{N-1}\left(
    \frac{\gamma_{N-1}}{\sigma_{N-1}}+\frac{\gamma_{N}}{\sigma_{N}}\right)
    -\frac{\hat{w}_{N-2}}{\sigma_{N-1}}\right)\sinh\left((x_{N}-x)\sqrt{s}
    \right),
\end{eqnarray*}
and thus the recurrence relations on the first interface is
\begin{equation}\label{5.10}
  \hat{w}_{1}^{k}=\hat{w}_{1}^{k-1}-\theta\left(\hat{w}_{1}^{k-1}
    \left(2+\frac{\gamma_{1}\gamma_{2}}{\sigma_{1}\sigma_{2}}
     +\frac{\sigma_{1}\gamma_{2}}{\gamma_{1}\sigma_{2}}\right)
    +\frac{\hat{w}_{2}^{k-1}}{\sigma_{2}}\left(
      \frac{\gamma_{3}}{\sigma_{3}}-\frac{\sigma_{1}}{\gamma_{1}}\right)
      -\frac{\hat{w}_{3}^{k-1}}{\sigma_{2}\sigma_{3}}\right),
\end{equation}
and on the last interface, we obtain
\begin{equation}\label{5.11}
  \begin{array}{rcl}
   \hat{w}_{N-1}^{k}&=&\displaystyle\hat{w}_{N-1}^{k-1}-\theta\left(
     \hat{w}_{N-1}^{k-1}\left(2+\frac{\gamma_{N-1}\gamma_{N}}{\sigma_{N-1}\sigma_{N}}
     +\frac{\sigma_{N}\gamma_{N-1}}{\gamma_{N}\sigma_{N-1}}\right)\right.\\
     &&\displaystyle\left.\hspace{5em}
     +\frac{\hat{w}_{N-2}^{k-1}}{\sigma_{N-1}}\left(
     \frac{\gamma_{N-2}}{\sigma_{N-2}}-\frac{\sigma_{N}}{\gamma_{N}}\right)
     -\frac{\hat{w}_{N-3}^{k-1}}{\sigma_{N-1}\sigma_{N-2}}\right).
  \end{array}
\end{equation}
Defining $\sigma:=\sinh(h_{\min}\sqrt{s})$ where $h_{\min}=\min_{1\leq
  i\leq N}h_{i}$, and setting
\begin{equation}\label{nusigmaw}
  \hat{\nu}_{i}^{k}(s):=\sigma^{2k}\hat{w}_{i}^{k}(s),
\end{equation}
relation (\ref{5.9}) reduces for the special choice $\theta=1/4$ to
\begin{equation}
  \hat{\nu}_{i}^{k}(s)=-\frac{1}{4}\left(
    \hat{t}_{i,i}\hat{\nu}_{i}^{k-1}(s)+\hat{t}_{i,i+1}\hat{\nu}_{i+1}^{k-1}(s)
    +\hat{t}_{i,i-1}\hat{\nu}_{i-1}^{k-1}(s)-\hat{t}_{i,i+2}\hat{\nu}_{i+2}^{k-1}(s)
    -\hat{t}_{i,i-2}\hat{\nu}_{i-2}^{k-1}(s)\right),
\end{equation}
where we defined
$\hat{t}_{i,i}:=\frac{2\sigma^{2}}{\sigma_{i}\sigma_{i+1}}(\gamma_{i}\gamma_{i+1}-\sigma_{i}\sigma_{i+1})$,
$\hat{t}_{i,i+1}:=\frac{\sigma^{2}}{\sigma_{i}\sigma_{i+1}\sigma_{i+2}}(\sigma_{i}\gamma_{i+2}-\gamma_{i}\sigma_{i+2})$,
$\hat{t}_{i,i-1}:=\frac{\sigma^{2}}{\sigma_{i}\sigma_{i-1}\sigma_{i+1}}(\sigma_{i+1}\gamma_{i-1}-\gamma_{i+1}\sigma_{i-1})$,
$\hat{t}_{i,i+2}:=\frac{\sigma^{2}}{\sigma_{i+1}\sigma_{i+2}}$, and
$\hat{t}_{i,i-2}:=\frac{\sigma^{2}}{\sigma_{i}\sigma_{i-1}}$. Similarly,
we obtain for (\ref{5.10})
\begin{equation}\label{LeftBE}
  \hat{\nu}_{1}^{k}(s)=-\frac{1}{4}\left(\hat{t}_{1,1}\hat{\nu}_{1}^{k-1}(s)
    +\hat{t}_{1,2}\hat{\nu}_{2}^{k-1}(s)-\hat{t}_{1,3}\hat{\nu}_{3}^{k-1}(s)\right),
\end{equation}
where we defined
$\hat{t}_{1,1}:=\sigma^{2}\left(\frac{\sigma_{1}\gamma_{2}}{\gamma_{1}\sigma_{2}}+\frac{\gamma_{1}\gamma_{2}}{\sigma_{1}\sigma_{2}}-2\right)$,
$\hat{t}_{1,2}=\frac{\sigma^{2}}{\sigma_{2}}\left(\frac{\gamma_{3}}{\sigma_{3}}-\frac{\sigma_{1}}{\gamma_{1}}\right)$
and $\hat{t}_{1,3}=\frac{\sigma^{2}}{\sigma_{2}\sigma_{3}}$. From (\ref{5.11}), we obtain
\begin{equation}\label{RightBE}
  \hat{\nu}_{N-1}^{k}(s)=-\frac{1}{4}\left(\hat{t}_{N-1,N-1}\hat{\nu}_{N-1}^{k-1}(s)
    +\hat{t}_{N-1,N-2}\hat{\nu}_{N-2}^{k-1}(s)-\hat{t}_{N-1,N-3}\hat{\nu}_{N-3}^{k-1}(s)\right),
\end{equation}
where we defined
$\hat{t}_{N-1,N-1}:=\sigma^{2}\left(\frac{\sigma_{N-1}\gamma_{N-2}}{\gamma_{N-1}\sigma_{N-2}}+\frac{\gamma_{N-1}\gamma_{N-2}}{\sigma_{N-1}\sigma_{N-2}}-2\right)$,
$\hat{t}_{N-1,N-2}=\frac{\sigma^{2}}{\sigma_{N-2}}\left(\frac{\gamma_{N-3}}{\sigma_{N-3}}-\frac{\sigma_{N-1}}{\gamma_{N-1}}\right)$
and $\hat{t}_{N-1,N-3}=\frac{\sigma^{2}}{\sigma_{N-2}\sigma_{N-3}}$.
Now we will show that for all $i$
\begin{equation}\label{nubound}
  \|\nu_{i}^{k}\|_{L^{\infty}(0,T)}\leq
    \frac{3}{2}\max_{1\leq j\leq N-1}
    \|\nu_{j}^{k-1}\|_{L^{\infty}(0,T)}, \quad k=1,2,3,\ldots.
\end{equation}
At first glance it seems that inequality (\ref{nubound}) is not enough
to prove convergence, because the factor $3/2$ is bigger than one, but
the $\nu_i^k$ are related to the $w_{i}^k$ by $\sigma$, see
(\ref{nusigmaw}), and the superlinear convergence of $w_i^k$ will come
from $\sigma$. To obtain (\ref{nubound}), we have to bound
$\int_{0}^{\infty}|t_{i,j}(t)|dt$, where
$t_{i,j}=\mathcal{L}^{-1}\left\{ \hat{t}_{i,j}\right\}$. We first
consider $\hat{t}_{i,i+2}$. If $h_{i+1}=h_{i+2}=h_{\min}$, then the
terms in $\hat{t}_{i,i+2}$ cancel and we simply get
$\hat{t}_{i,i+2}=1$. If $h_{\min}\leq h_{i+1}$ and $h_{\min}\leq
h_{i+2}$, then the kernel $t_{i,i+2}$, being a convolution of
two positive functions, is positive by part \ref{L1} of Lemma
\ref{SimpleLaplaceLemma}, and using Lemma \ref{LimitLemma} its
integral is bounded by
\[
  \int_{0}^{\infty}|t_{i,i+2}(t)|dt\le
  \lim_{s\rightarrow0+}\frac{\sinh^{2}(h_{\min}\sqrt{s})}
      {\sinh(h_{i+1}\sqrt{s})\sinh(h_{i+2}\sqrt{s})}
      \leq\frac{h_{\min}^{2}}{h_{i+1}h_{i+2}}\leq1,
\]
so that
\[
  \|\mathcal{L}^{-1}\left(
  \hat{t}_{i,i+2}\hat{\nu}_{i+2}^{k}(s)\right)
   \|_{L^{\infty}(0,T)}\leq\|\nu_{i+2}^{k}\|_{L^{\infty}(0,T)}.
\]
The same argument also holds for the term involving
$\hat{\nu}_{i-2}^{k-1}$.  Now for $\hat{\nu}_{i+1}^{k-1}$, we rewrite
$\hat{t}_{i,i+1}$ as
\[
  \hat{t}_{i,i+1}=\frac{\sinh\left((h_{i}-h_{i+2})\sqrt{s}\right)
    \sinh^{2}(h_{\min}\sqrt{s})}
    {\sinh(h_{i}\sqrt{s})\sinh(h_{i+1}\sqrt{s})\sinh(h_{i+2}\sqrt{s})}.
\]
Assuming that $h_{i}\geq h_{i+2}$, we use Lemma \ref{PositivityLemma}
and Lemma \ref{LimitLemma} to get a bound of the form
\[
  \int_{0}^{\infty}|t_{i,i+1}(t)|dt\le 
  \frac{|h_{i}-h_{i+2}|h_{\min}^{2}}{h_{i}h_{i+1}h_{i+2}}<1,
\]
and similarly for the term involving $\hat{\nu}_{i-1}^{k-1}$. Finally,
for the term $\hat{t}_{i,i}$, we use the trigonometric identity
$\sinh(A)\cosh(B)=\frac{1}{2}\left(\sinh(A+B)+\sinh(A-B)\right)$
to obtain
\begin{eqnarray*}
  \hat{t}_{i,i}&=&\frac{\sinh(h_{\min}\sqrt{s})\sinh\left(
    (h_{\min}+h_{i}-h_{i+1})\sqrt{s}\right)}
    {\sinh(h_{i}\sqrt{s})\sinh(h_{i+1}\sqrt{s})}\\
  &+&\frac{\sinh(h_{\min}\sqrt{s})\sinh\left( 
    (h_{\min}+h_{i+1}-h_{i})\sqrt{s}\right)}
    {\sinh(h_{i}\sqrt{s})\sinh(h_{i+1}\sqrt{s})}.
\end{eqnarray*}
Each term is again a ratio of hyperbolic sines, so we only need to
pair the factors so that the coefficient in the numerator is always
smaller than the one in the denominator. Now $-h_{i+1}\leq
h_{\min}+h_{i}-h_{i+1}=h_{i}+h_{\min}-h_{i+1}\leq h_{i}$, so for the
first term, we choose the pairing
\[
  \frac{\sinh(h_{\min}\sqrt{s})}{\sinh(h_{i+1}\sqrt{s})}\cdot
  \frac{\sinh\left((h_{\min}+h_{i}-h_{i+1})\sqrt{s}\right)}
   {\sinh(h_{i}\sqrt{s})}\quad\textrm{if}\; h_{i+1}\leq h_{i},
\]
and
\[
  \frac{\sinh(h_{\min}\sqrt{s})}{\sinh(h_{i}\sqrt{s})}\cdot
  \frac{\sinh\left((h_{\min}+h_{i}-h_{i+1})\sqrt{s}\right)}
  {\sinh(h_{i+1}\sqrt{s})}\quad\textrm{if}\; h_{i+1}\geq h_{i}.
\]
A similar argument holds also for the second term. Now using Lemma
\ref{LimitLemma}, we have again integrals of kernels bounded by 1. In
summary, we get for $2\leq i\leq N-2$
\begin{eqnarray*}
  \|\nu_{i}^{k}\|_{L^{\infty}(0,T)} & \leq & 
    \frac{1}{2}\|\nu_{i}^{k-1}\|_{L^{\infty}(0,T)}
   +\frac{1}{4}\left(\|\nu_{i-2}^{k-1}\|_{L^{\infty}(0,T)}
   +\|\nu_{i-1}^{k-1}\|_{L^{\infty}(0,T)}\right.\\
   &&\hspace{9em} +\left.\|\nu_{i+1}^{k-1}\|_{L^{\infty}(0,T)}
   +\|\nu_{i+2}^{k-1}\|_{L^{\infty}(0,T)}\right),
\end{eqnarray*}
and the estimate (\ref{nubound}) is established for interior
subdomains. For the left subdomain touching the boundary, 
the kernel $t_{1,3}$ in (\ref{LeftBE}) can 
be estimated like $t_{i,i+2}$. For $\hat{t}_{1,2}$, we
have 
\[
  \hat{t}_{1,2}=\frac{\cosh\left((h_{1}-h_{3})\sqrt{s}\right)
    \sinh^{2}(h_{\min}\sqrt{s})}
    {\cosh(h_{1}\sqrt{s})\sinh(h_{2}\sqrt{s})\sinh(h_{3}\sqrt{s})}.
\]
If $h_{1}\geq h_{3}$, the decomposition
$\hat{t}_{1,2}=\frac{\cosh\left((h_{1}-h_{3})\sqrt{s}\right)}
{\cosh(h_{1}\sqrt{s})}
\cdot\frac{\sinh(h_{\min}\sqrt{s})}{\sinh(h_{2}\sqrt{s})}
\cdot\frac{\sinh(h_{\min}\sqrt{s})}{\sinh(h_{3}\sqrt{s})}$ shows that
one can bound
\[
  \int_{0}^{\infty}|t_{1,2}(t)|dt\le
    h_{\min}^{2}/h_{2}h_{3}<1.
\]
If $h_{3}>h_{1}$, then we rewrite
\begin{eqnarray*}
  \hat{t}_{1,2}&=&\frac{1}{\cosh(h_{1}\sqrt{s})}\cdot
    \frac{\sinh(h_{\min}\sqrt{s})}{\sinh(h_{2}\sqrt{s})}\left(\frac{\sinh
      \left((h_{\min}+h_{1}-h_{3})\sqrt{s}\right)}
      {2\sinh(h_{3}\sqrt{s})}\right.\\
  &&\hspace{13em}\left.+\frac{\sinh\left((h_{\min}+h_{3}-h_{1})\sqrt{s}\right)}
   {2\sinh(h_{3}\sqrt{s})}\right),
\end{eqnarray*}
which again shows using Lemma \ref{LimitLemma} that the integral is
bounded by 1.  Finally we consider
\begin{eqnarray*}
\hat{t}_{1,1} & = & \frac{2\cosh\left((2h_{1}-h_{2})\sqrt{s}\right)\sinh^{2}(h_{\min}\sqrt{s})}{\sinh(2h_{1}\sqrt{s})\sinh(h_{2}\sqrt{s})}\\
 & = & \frac{\sinh(h_{\min}\sqrt{s})\sinh\left ((h_{\min}+2h_{1}-h_{2})\sqrt{s}\right) }{\sinh(2h_{1}\sqrt{s})\sinh(h_{2}\sqrt{s})}+\frac{\sinh(h_{\min}\sqrt{s})\sinh\left( (h_{\min}-2h_{1}+h_{2})\sqrt{s}\right) }{\sinh(2h_{1}\sqrt{s})\sinh(h_{2}\sqrt{s})}.
\end{eqnarray*}
Using the inequalities $-h_{2}\leq h_{\min}+2h_{1}-h_{2}\leq2h_{1}$
and $-2h_{1}\leq h_{\min}-2h_{1}+h_{2}\leq h_{2}$ we can again, with an
appropriate pairing of factors, bound the integral of each term
by 1. Thus we have for the subdomain touching the left physical boundary
\[
\|\nu_{1}^{k}\|_{L^{\infty}(0,T)}\leq\frac{1}{2}\|\nu_{1}^{k-1}\|_{L^{\infty}(0,T)}+\frac{1}{4}\left(\|\nu_{2}^{k-1}\|_{L^{\infty}(0,T)}+\|\nu_{3}^{k-1}\|_{L^{\infty}(0,T)}\right).
\]
A similar result holds for $\hat{\nu}_{N-1}^{k}(s)$, and hence the
inequality (\ref{nubound}) holds for all $1\leq i\leq N-1$. Therefore,
by induction, we obtain 
\[
  \max_{1\leq j\leq N-1}\|\nu_{j}^{k}\|_{L^{\infty}(0,T)}
    \leq\left(\frac{3}{2}\right)^{k}\max_{1\leq j\leq N-1}
    \|\nu_{j}^{0}\|_{L^{\infty}(0,T)}
   =\left(\frac{3}{2}\right)^{k}\max_{1\leq j\leq N-1}
   \| w_{j}^{0}\|_{L^{\infty}(0,T)}.
\]
Now since
\[
w_{i}^{k}(t)=\left(\phi_{2k}*\nu_{i}^{k}\right)(t)=\int_{0}^{t}\phi_{2k}(t-\tau)\nu_{i}^{k}(\tau)d\tau,
\]
with $\phi_{2k}=\mathcal{L}^{-1}
\left(\frac{1}{\sinh^{2k}(h_{\min}\sqrt{s})}\right)$,
a similar estimate using part \ref{L3} of Lemma
\ref{SimpleLaplaceLemma} as in the proof of Theorem \ref{Theorem3}
leads to the superlinear convergence estimate (\ref{Theorem4Result}).

\section{Analysis of the NNWR in 2D}\label{Section6}

In this section we formulate and analyze the NNWR algorithm, applied
to the two-dimensional heat equation
\[
  \partial_t u-\Delta u=f(x,y,t),\quad
    (x,y)\in\Omega=(l,L)\times(0,\pi),\: t\in(0,T]
\]
with initial condition $u(x,y,0)=u_0(x,y)$ and Dirichlet boundary conditions. To define the Neumann-Neumann
algorithm, we decompose $\Omega$ into strips of the form
$\Omega_{i}=(x_{i-1},x_{i})\times(0,\pi)$,
$l=x_{0}<x_{1}<\cdots<x_{N}=L$. The Neumann-Neumann algorithm,
considering directly the error equations with $f(x,y,t)=0,u_0(x,y)=0$ and
homogeneous Dirichlet boundary conditions, is then given by performing
iteratively for $k=1,2,\ldots$ and for $i=1,\ldots,N$ the Dirichlet
and Neumann steps
\begin{equation}\label{NNWR2D}
  \arraycolsep0.2em
  \begin{array}{rclrcl}
  \partial_t u_{i}^{k}-\Delta u_{i}^{k}&=&0,
    \qquad \textrm{in $\Omega_{i}$}, & 
    \partial_t\psi_{i}^k-\Delta\psi_{i}^k & = & 
    0,\qquad\textrm{in $\Omega_{i}$},\\
    u_{i}^{k}(x,y,0)&=&0, & 
    \psi_{i}^k(x,y,0) & = & 0,\\
    u_{i}^k(x_{i-1},y,t) & = & g_{i-1}^{k-1}(y,t),&
    \partial_{n_i}\psi_{i}^k(x_{i-1},y,t) & = & 
      (\partial_{n_{i-1}} u_{i-1}^k
      +\partial_{n_i} u_{i}^k)(x_{i-1},y,t),\\
   u_{i}^k(x_{i},y,t) & = & g_{i}^{k-1}(y,t),&
    \partial_{n_i}\psi_{i}^k(x_{i},y,t)&=&
     (\partial_{n_{i-1}} u_{i-1}^k
     +\partial_{n_i} u_{i}^k)(x_{i},y,t),\\
   u_{i}^k(x,0,t)&=&u_{i}^k(x,\pi,t)=0,&\psi_{i}^k(x,0,t)&=&\psi_{i}^k(x,\pi,t)=0,
  \end{array}
\end{equation}
except for the first and last subdomain, where in the Neumann step the
Neumann conditions are replaced by homogeneous Dirichlet conditions
along the physical boundaries, as in the one dimensional case. The new
interface values for the next step are then defined as
\[
  g_{i}^{k}(y,t)=g_{i}^{k-1}(y,t)
   -\theta\left(\psi_{i}^k(x_{i},y,t)+\psi_{i+1}^k(x_{i},y,t)\right) .
\]
To analyze the NNWR algorithm (\ref{NNWR2D}) in two dimensions, we
first reduce the problem to a collection of one-dimensional problems
by performing a Fourier transform along the $y$ direction. 
More precisely, we use a Fourier sine series along the
$y$-direction,
\[
  u_{i}^k(x,y,t)=\sum_{n\geq1}U_{i}^k(x,n,t)\sin(ny)
\]
where
\[
  U_{i}^k(x,n,t)=\frac{2}{\pi}\int_{0}^{\pi}u_{i}^k(x,\eta,t)\sin(n\eta)d\eta.
\]
Thus, the 2D problems in the NNWR algorithm (\ref{NNWR2D}) become a
sequence of 1D problems indexed by $n$,
\begin{equation}\label{NNWRDFT}
  \frac{\partial U_{i}^k}{\partial t}(x,n,t)
    -\frac{\partial^{2}U_{i}^k}{\partial x^{2}}(x,n,t)+n^{2}U_{i}^k(x,n,t)=0,
\end{equation}
and the boundary conditions for $U_{i}^k(x,n,t)$ are identical to the
one-dimensional case for each $n$. 
\begin{theorem}\label{2DTheorem}
\textbf{(Convergence of NNWR in 2D)} Let $\theta=1/4$. For $T > 0$ fixed, the NNWR algorithm (\ref{NNWR2D}) converges 
superlinearly with the estimate
\[
  \max_{1\leq i\leq N-1}\| g_{i}^{k}
    \|_{L^{\infty}\left(0,T;L^{2}(0,\pi)\right)}
    \leq\left(\frac{\sqrt{6}}{1-e^{-\frac{(2k+1)h_{\min}^{2}}{T}}}
    \right)^{2k}e^{-k^{2}h_{\min}^{2}/T}
    \max_{1\leq i\leq N-1}\| g_{i}^{0}
    \|_{L^{\infty}\left(0,T;L^{2}(0,\pi)\right)},
\]
where $h_{\min}$ is the minimum subdomain width.
\end{theorem}
\begin{proof}
We take Laplace transforms in $t$ of (\ref{NNWRDFT}) to get
\[
  (s+n^{2})\hat{U}_{i}^k-\frac{d^{2}\hat{U}_{i}^k}{dx^{2}}=0,
\]
and now treat each $n$ as in the one-dimensional analysis in the proof
of Theorem \ref{Theorem4}, where the recurrence relations (\ref{5.9}),
(\ref{5.10}) and (\ref{5.11}) of the form 
\[
  \hat{w}_{i}^{k}(s)=\sum_{j}A_{ij}^{(k)}(s)\hat{w}_{j}^{0}(s)
\]
now become for each $n=1,2,\ldots$
\begin{equation}\label{6.2}
  \hat{G}_{i}^{k}(n,s)=\sum_{j}A_{ij}^{(k)}(s+n^{2})\hat{G}_{j}^{0}(n,s).
\end{equation}
If $a_{ij}^{(k)}(t)$ is the inverse Laplace transform of
$A_{ij}^{(k)}(s)$, i.e.,
\begin{equation}\label{6.3}
A_{ij}^{(k)}(s)={\displaystyle \int_{0}^{\infty}}a_{ij}^{(k)}(t)e^{-st}dt,
\end{equation}
then if we replace $s$ by $s+n^{2}$ in (\ref{6.3}), we get
$A_{ij}^{(k)}(s+n^{2})=\int_{0}^{\infty}a_{ij}^{(k)}(t)e^{-n^{2}t}e^{-st}dt,$
so the inverse Laplace transform of $A_{ij}^{(k)}(s+n^{2})$ is just
$a_{ij}^{(k)}(t)e^{-n^{2}t}$. Hence taking the inverse
Laplace transform of (\ref{6.2}), we get
\[
  G_{i}^{k}(n,t)=\sum_{j} \int_{0}^{t}a_{ij}^{(k)}(\tau)
    e^{-n^{2}\tau}G_{j}^{0}(n,t-\tau)d\tau.
\]
So the interface functions $g_{i}^{k}(y,t)$ can be written as
\begin{eqnarray*}
g_{i}^{k}(y,t) & = & {\displaystyle \sum_{n\geq 1}}G_{i}^{k}(n,t)\sin(ny)\\
 & = & {\displaystyle \sum_{n\geq 1}}{\displaystyle \sum_{j}}{\displaystyle \int_{0}^{t}}a_{ij}^{(k)}(\tau)e^{-n^{2}\tau}\left(\frac{2}{\pi}{\displaystyle \int_{0}^{\pi}}g_{j}^{0}(\eta,t-\tau)\sin(n\eta)d\eta\right)\sin(ny)d\tau.
\end{eqnarray*}
Next, we justify the exchange of the infinite sum and the integrals using Fubini's theorem.
Here, we need to check that
$ \left|\sum_{n=1}^\infty a_{ij}^{(k)}(\tau) e^{-n^2\tau}\right| $
remains bounded for all $\tau \geq 0$. For $\tau$ bounded away from zero, this follows from the boundedness of $|a_{ij}^{(k)}|$ and from the geometric series, so it suffices to show boundedness
for $\tau$ close to zero. 
To do so, note that $A_{ij}^{(k)}(s)$ contains $1/\sinh^{2k}(h_{\min}\sqrt{s})$ as a factor,
which implies $\lim_{s\to\infty} s^p A_{ij}^{k}(s) = 0$ for all $p > 1$. This means $a_{ij}^{(k)}(\tau)$
is infinitely differentiable at $\tau = 0$ and its derivatives of all orders vanish there. Thus, by Taylor's theorem, there exists a constant $C$ such that $|a_{ij}^{(k)}(\tau)| \leq C\tau^2$ for $\tau > 0$ small enough, so we have 
\begin{equation}\label{eq6.6a}
\left|\sum_{n=1}^{M} a_{ij}^{(k)}(\tau)e^{-n^{2}\tau}\right|\leq 
\frac{C\tau^2}{1-e^{-\tau}}.
\end{equation}
In particular, for $0 < \tau < 1$, we have 
$ 1 - e^{-\tau} \geq \tau - \frac{\tau^2}{2} \geq \frac{\tau}{2}, $
so the sum \eqref{eq6.6a} is bounded above by $2C\tau$, 
which is independent of $M$.  Therefore, $ \left|\sum_{n=1}^\infty a_{ij}^{(k)}(\tau) e^{-n^2\tau}\right| $ is bounded uniformly for all $\tau \in (0,\infty)$, so we can apply
Fubini's theorem to interchange sums and integrals and get
\begin{equation}\label{6.4}
  g_{i}^{k}(y,t)= \sum_{j}\int_{0}^{t}a_{ij}^{(k)}(\tau)
    \int_{0}^{\pi}g_{j}^{0}(\eta,t-\tau)\left(\frac{2}{\pi}
    \sum_{n\geq1}e^{-n^{2}\tau}\sin(n\eta)\sin(ny)\right)d\eta d\tau.
\end{equation}
 We now use the trigonometric identity
 $\sin(A)\sin(B)=\frac{1}{2}\cos(A-B)-\frac{1}{2}\cos(A+B)$ to rewrite
 (\ref{6.4}) as
\begin{eqnarray*}
  \frac{2}{\pi}\sum_{n\geq1}e^{-n^2\tau}\sin(n\eta)\sin(ny)
  &=&\frac{1}{\pi} \sum_{n\geq 1}e^{-n^{2}\tau}\left(
    \cos\left(n(\eta-y)\right)-\cos\left(n(\eta+y)\right)\right)\\
  &=&\frac{1}{2\pi}\sum_{n\in\mathbb{Z}}e^{-n^{2}\tau}\left( 
    \exp\left(in(\eta-y)\right)-\exp\left(in(\eta+y)\right)\right) .
\end{eqnarray*}
Now we recall the following well-known properties of the Fourier transform
$\hat{f}(w)={\mathcal{F}}(f(x))(w):=\int_{-\infty}^{\infty}f(t)e^{-iwt}dt$:
\begin{enumerate}
\item \label{P1} if both $f$ and $\hat{f}$ are continuous and decay
  sufficiently rapidly, then $\sum_{n\in\mathbb{Z}}f(n)=
  \sum_{k\in\mathbb{Z}}\hat{f}(2k\pi)$ (Poisson summation formula),

\item \label{P2} $\mathcal{F}\left( f(x)e^{iw_{0}x}\right) =\hat{f}(w-w_{0})$,

\item \label{P3} $\mathcal{F}\left(e^{-x^{2}\tau}\right)=\sqrt{\frac{\pi}{\tau}}
   e^{-w^{2}/4\tau}$.
\end{enumerate}
Thus, using the properties \ref{P2} and \ref{P3}  and the Poisson summation formula, we obtain
\[
  \frac{2}{\pi}\sum_{n\geq1}e^{-n^{2}\tau}\sin(n\eta)\sin(ny)
    =\frac{1}{\sqrt{4\pi\tau}} \sum_{k\in\mathbb{Z}}
    \left(e^{-(2k\pi-\eta+y)^{2}/4\tau}-e^{-(2k\pi-\eta-y)^{2}/4\tau}\right).
\]
Interchanging the sum and the integral, (\ref{6.4}) gives
\begin{equation}\label{6.5}
  g_{i}^{k}(y,t)=\sum_{j}\int_{0}^{t}\frac{a_{ij}^{(k)}(\tau)}{\sqrt{4\pi\tau}}
  \left(\sum_{k\in\mathbb{Z}} \int_{0}^{\pi}g_{j}^{0}(\eta,t-\tau)
    \left(e^{-(y-(\eta-2k\pi))^{2}/4\tau}-e^{-(y+(\eta-2k\pi))^{2}/4\tau}\right)
    d\eta\right) d\tau.
\end{equation}
Now splitting the two integrals and performing the change of variables
$\zeta=\eta-2k\pi$ in the first integral and $\zeta=2k\pi-\eta$
in the second, (\ref{6.5}) gives
\[
  \sum_{k\in\mathbb{Z}}\int_{-2k\pi}^{(1-2k)\pi}
    g_{j}^{0}(\zeta+2k\pi,t-\tau)e^{-(y-\zeta)^{2}/4\tau}d\zeta
   -\sum_{k\in\mathbb{Z}} \int_{(2k-1)\pi}^{2k\pi}
    g_{j}^{0}(2k\pi-\zeta,t-\tau)e^{-(y-\zeta)^{2}/4\tau}d\zeta.
\]
Letting $m=-k$ in the first integral, we obtain
\[
  \sum_{m\in\mathbb{Z}}\int_{2m\pi}^{(2m+1)\pi}
    g_{j}^{0}(\zeta-2m\pi,t-\tau)e^{-(y-\zeta)^{2}/4\tau}d\zeta
    -\sum_{m\in\mathbb{Z}} \int_{(2m-1)\pi}^{2m\pi}
   g_{j}^{0}(2m\pi-\zeta,t-\tau)e^{-(y-\zeta)^{2}/4\tau}d\zeta.
\]
Defining the $2\pi$-periodic odd extension of $g_{j}^{0}$ as
\[
 \bar{g}_{j}^{0}(y,t)=\begin{cases}
   g_{j}^{0}(y-2m\pi,t), & 2m\pi<y<(2m+1)\pi,\\
    -g_{j}^{0}(2m\pi-y,t), & (2m-1)\pi<y<2m\pi\quad(m\in\mathbb{Z}),
  \end{cases}
\]
we can rewrite (\ref{6.5}) as
\begin{equation}
  g_{i}^{k}(y,t)=\sum_{j}\int_{0}^{t}\int_{-\infty}^{\infty}
    \frac{a_{ij}^{(k)}(\tau)}{\sqrt{4\pi\tau}}
    \bar{g}_{j}^{0}(\zeta,t-\tau)e^{-(y-\zeta)^{2}/4\tau}d\zeta d\tau.
\end{equation}
Now since $\bar{g}_{j}^{0}$
and $g_{j}^{0}$ have the same maxima and minima, we have 
\begin{equation}
  |g_{i}^{k}(y,t)|\leq \sum_{j}\| g_{j}^{0}\|
    \int_{0}^{t} \int_{-\infty}^{\infty}|a_{ij}^{(k)}(\tau)|
    \frac{1}{\sqrt{4\pi\tau}}e^{-(y-\zeta)^{2}/4\tau}d\zeta d\tau,
\end{equation}
where $\| g_{j}^{0}\|=\max_{0<y<\pi}
\max_{0<t<T}|g_{j}^{0}(y,t)|$ is the $L^{\infty}$ norm of the initial
guess. Also note
\[
  \int_{-\infty}^{\infty}\frac{1}{\sqrt{4\pi\tau}}
    e^{-(y-\zeta)^{2}/4\tau}d\zeta
    =\frac{1}{\sqrt{4\pi\tau}} \int_{-\infty}^{\infty}
    e^{-\zeta{}^{2}/4\tau}d\zeta=1,
\]
so that
\[
  \| g_{i}^{k}\|\leq \sum_{j}\|
  g_{j}^{0}\| \int_{0}^{t}|a_{ij}^{(k)}(\tau)|d\tau,
\]
which means we have the same bounds as in the 1D case.
\end{proof}

\section{Numerical Experiments}\label{Section7}

We perform experiments to measure the actual convergence rate of the
discretized DNWR and NNWR algorithms for the problem
\begin{equation}\label{NumericalModelProblem}
\begin{array}{rcll}
  \partial_t u-\frac{\partial}{\partial x}\left(\kappa(x)\partial_{x}u\right)&=&0, \qquad& x\in\Omega,\\
  u(x,0)&=&x(x+1)(x+3)(x-2)e^{-x}, & x\in\Omega,\\
  u(-3,t)=t,\;u(2,t)&=&te^{-t}, & t>0.
\end{array}
\end{equation}
We discretize (\ref{NumericalModelProblem}) using standard centered
finite differences in space and backward Euler in time on a grid with
$\Delta x=2\times 10^{-2}$ and $\Delta t=4\times 10^{-3}$. For the
DNWR method, we consider two cases: first, we choose $a=3$ and $b=2$,
i.e., we split the spatial domain $\Omega:=(-3,2)$ into two
non-overlapping subdomains $\Omega_1=(-3,0)$ and $\Omega_2=(0,2)$, see
Figure \ref{FigDecomp}.  This is the case of DNWR when the Dirichlet
subdomain is larger than the Neumann subdomain ($a>b$), corresponding
to Theorem \ref{Theorem2}. For the second case, we take $a=2$ and
$b=3$, so that the Dirichlet domain is smaller than the Neumann one,
as in Theorem \ref{Theorem3}. We test the algorithm by choosing
$h^{0}(t)=t^2, t\in[0,T]$ as an initial guess. Figures \ref{NumFig01}
and \ref{NumFig02}
\begin{figure}
  \centering
  \includegraphics[width=0.49\textwidth]{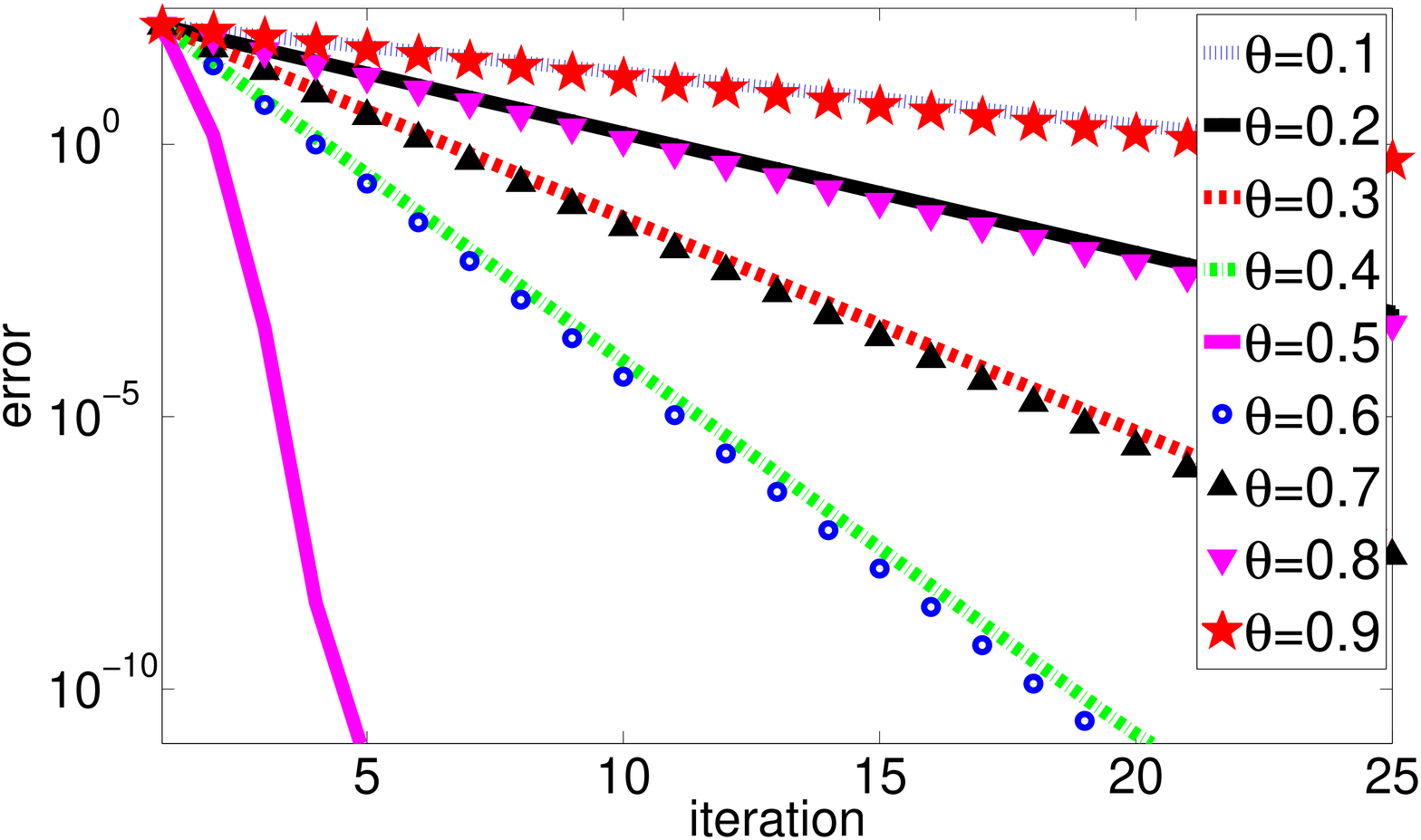}
  \includegraphics[width=0.49\textwidth]{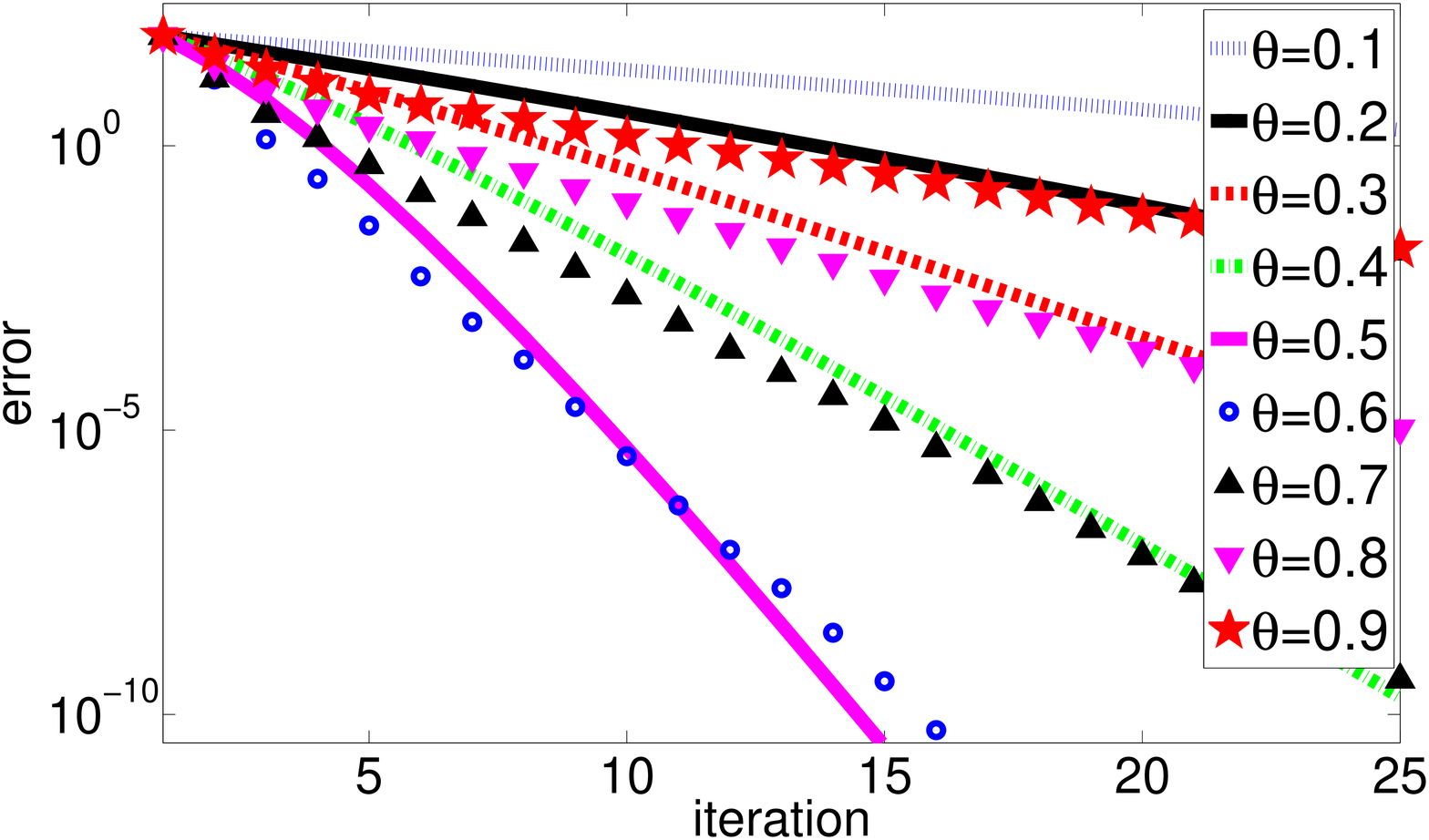}
  \caption{Convergence of DNWR for $a>b$ using various relaxation
    parameters $\theta$ for $T=2$, on the left for $\kappa(x)=1$ and on the right for
    $\kappa(x)=1+e^x$}
  \label{NumFig01}
\end{figure}
\begin{figure}
  \centering
  \includegraphics[width=0.49\textwidth]{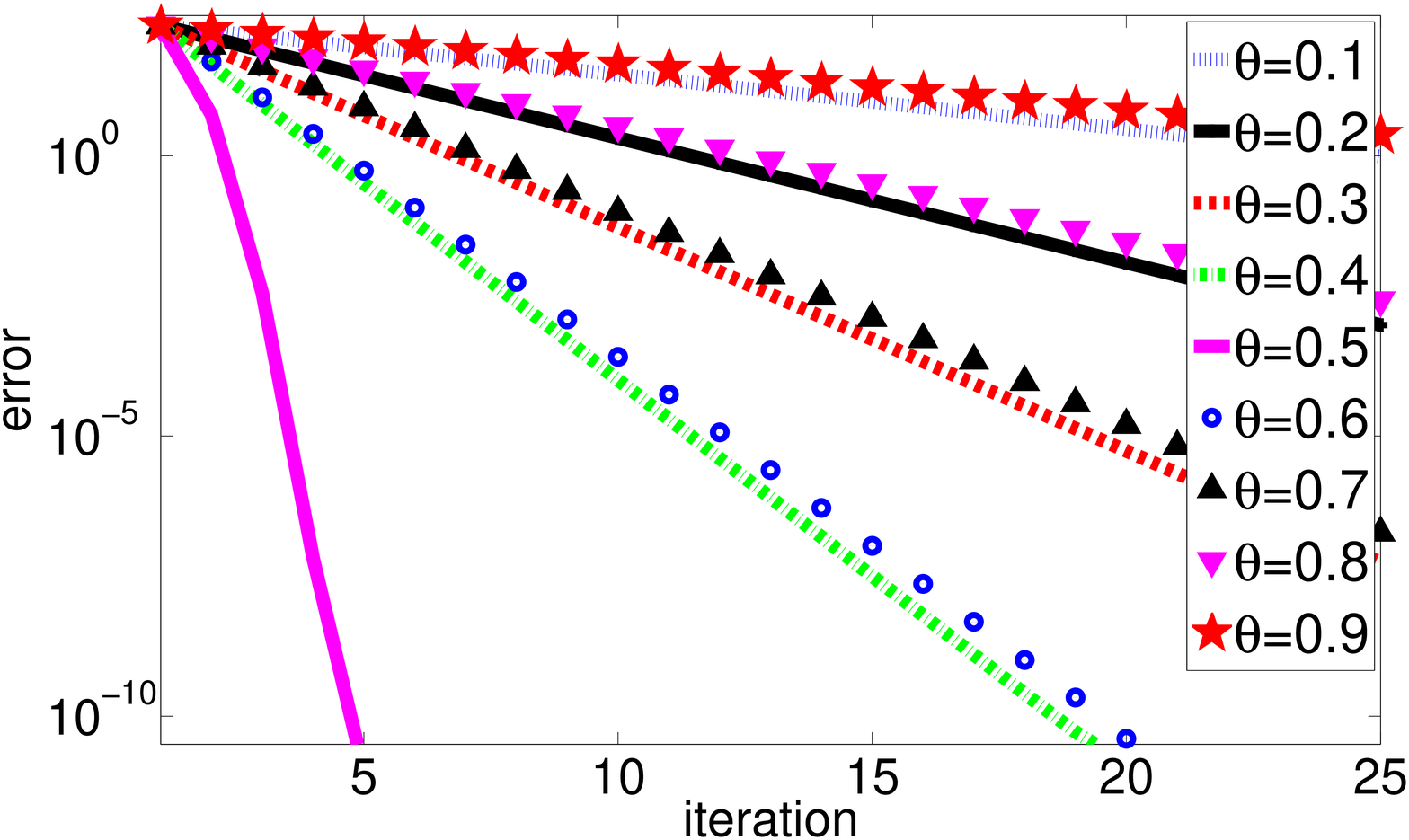}
  \includegraphics[width=0.49\textwidth]{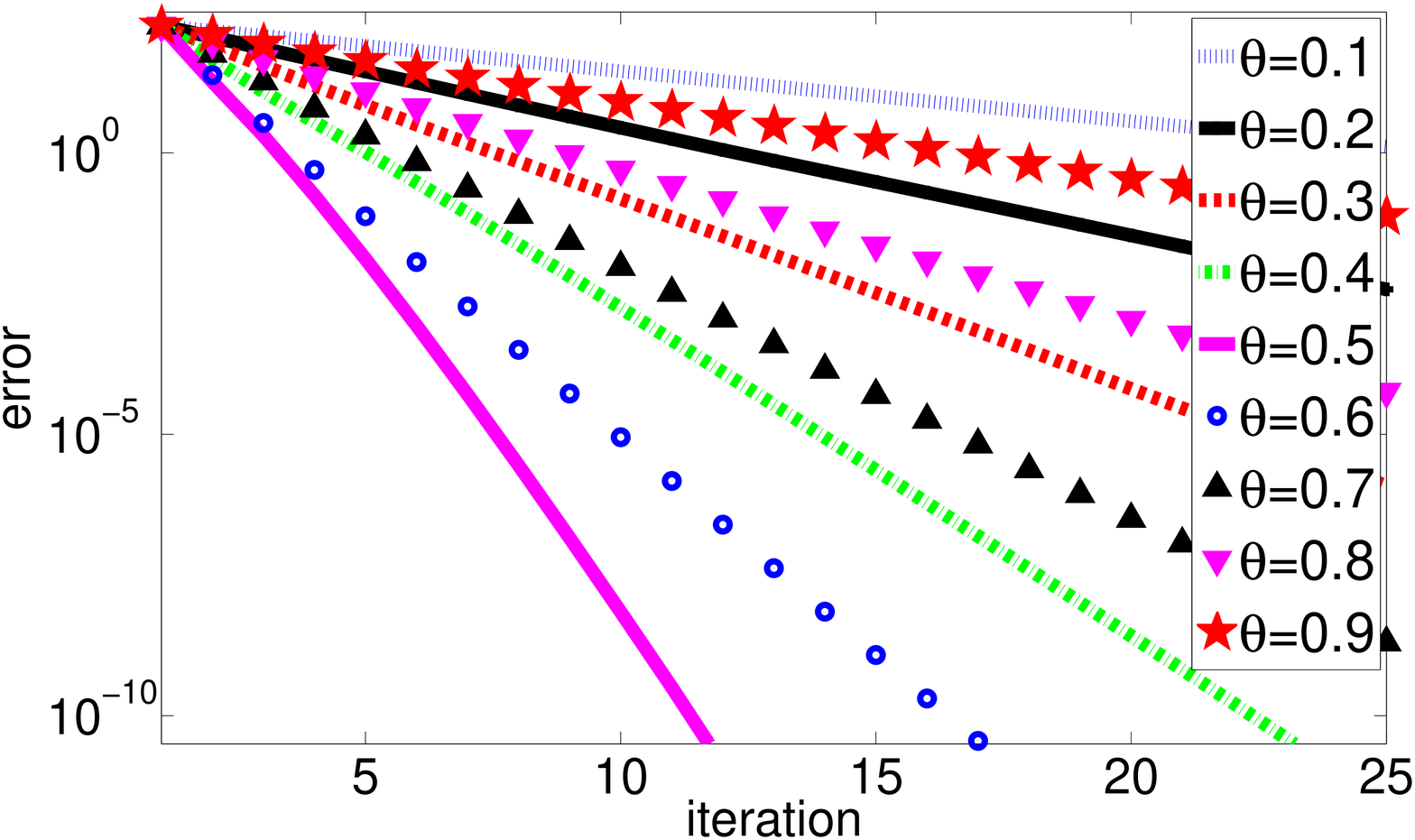}
  \caption{Convergence of DNWR for $a<b$ using various relaxation
    parameters $\theta$ for $T=2$, on the left for $\kappa(x)=1$ and on the right for
    $\kappa(x)=1+e^x$}
  \label{NumFig02}
\end{figure}
give the convergence curves for $T=2$ and for different values of the
parameter $\theta$ for $\kappa(x)=1$ on the left, and
$\kappa(x)=1+e^x$ on the right. We see that for a small time window,
we get linear convergence for all relaxation parameters $\theta$,
except for $\theta=1/2$, when we observe superlinear convergence. 

%For a large time window, we always observe
%linear convergence for $\theta=0.5$.
%We always use $\kappa=3$, unless otherwise specified. 

Next, we show an experiment for the NNWR algorithm in the spatial
domain $\Omega=(0,6)$, with the same discretization parameters $\Delta
x$ and $\Delta t$ as above, and for the time window $T=2$. Now onward
we always use $\kappa(x)=1$, unless otherwise specified. In Figure
\ref{NumFig2}, we consider a decomposition into two to six unequal
subdomains, whose widths are shown in Table \ref{Table1}. On the left
panel, we show the convergence in the four-subdomain case as a
function of the relaxation parameter $\theta$, whereas on the right
panel, we show the convergence for $\theta = 1/4$ as we vary the
number of subdomains. We observe superlinear convergence for
$\theta=1/4$, and only linear convergence for the other choices, and
also that convergence slows down as the number of subdomains is
increased, as expected.

\begin {table}
\begin{center}
\caption {Subdomain lengths used for the NNWR experiments in Fig.\ref{NumFig2}. \label{Table1}}
\begin{tabular}{|c|c|c|c|c|c|c|}
\hline 
No. of subdomains & $h_{1}$ & $h_{2}$ & $h_{3}$ & $h_{4}$ & $h_{5}$ & $h_{6}$\tabularnewline
\hline 
2 & 3.50 & 2.50 &  &  &  & \tabularnewline
\hline 
3 & 2.30 & 2.30 & 1.40 &  &  & \tabularnewline
\hline 
4 & 1.20 & 2.40 & 1.80 & 0.60 &  & \tabularnewline
\hline 
5 & 1.80 & 1.40 & 1.08 & 1.00 & 0.72 & \tabularnewline
\hline 
6 & 1.20 & 0.80 & 1.00 & 1.20 & 1.00 & 0.80\tabularnewline
\hline 
\end{tabular}
\end{center}
\end {table}

\begin{figure}
  \centering
  \includegraphics[width=0.49\textwidth]{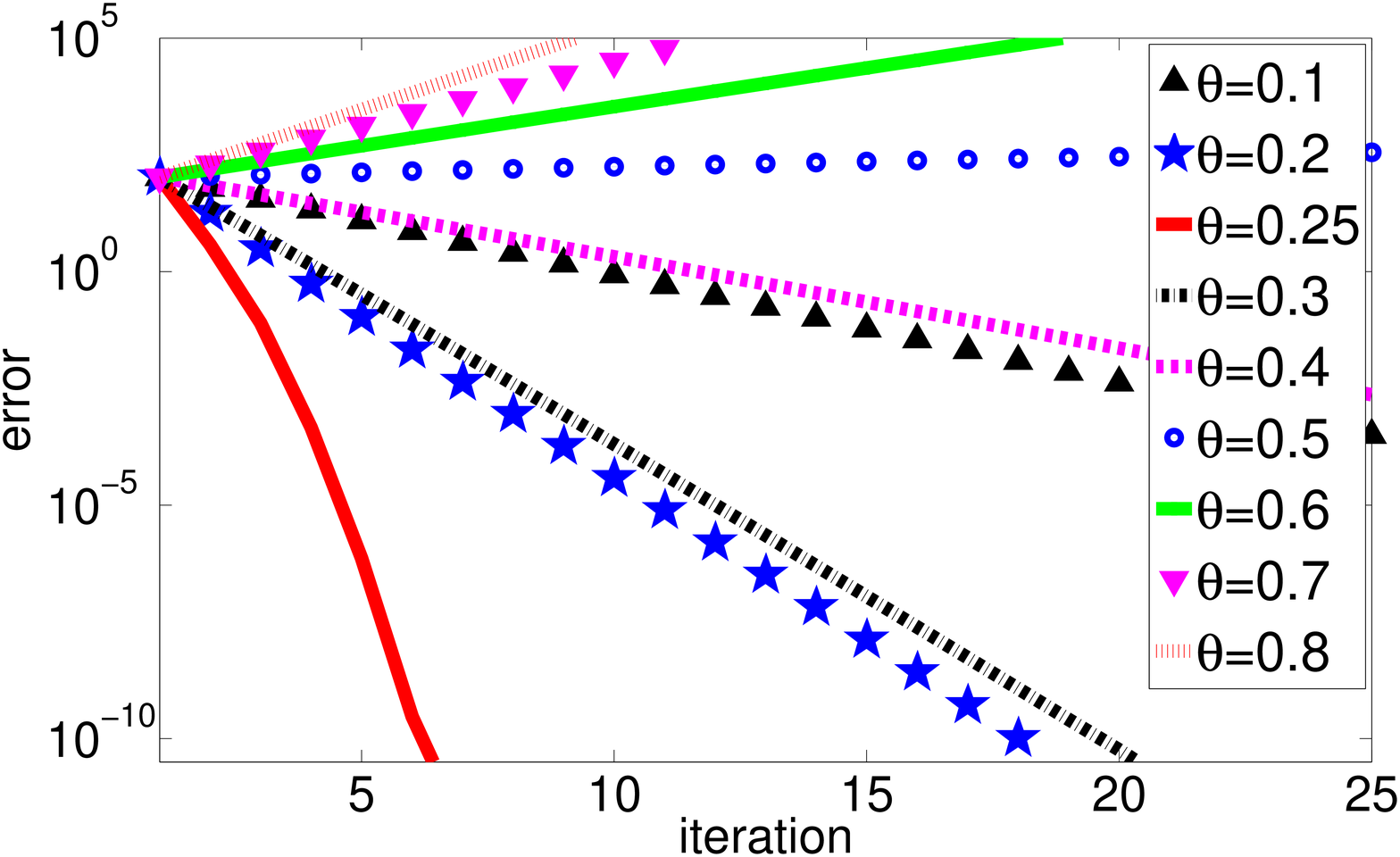}
  \includegraphics[width=0.49\textwidth]{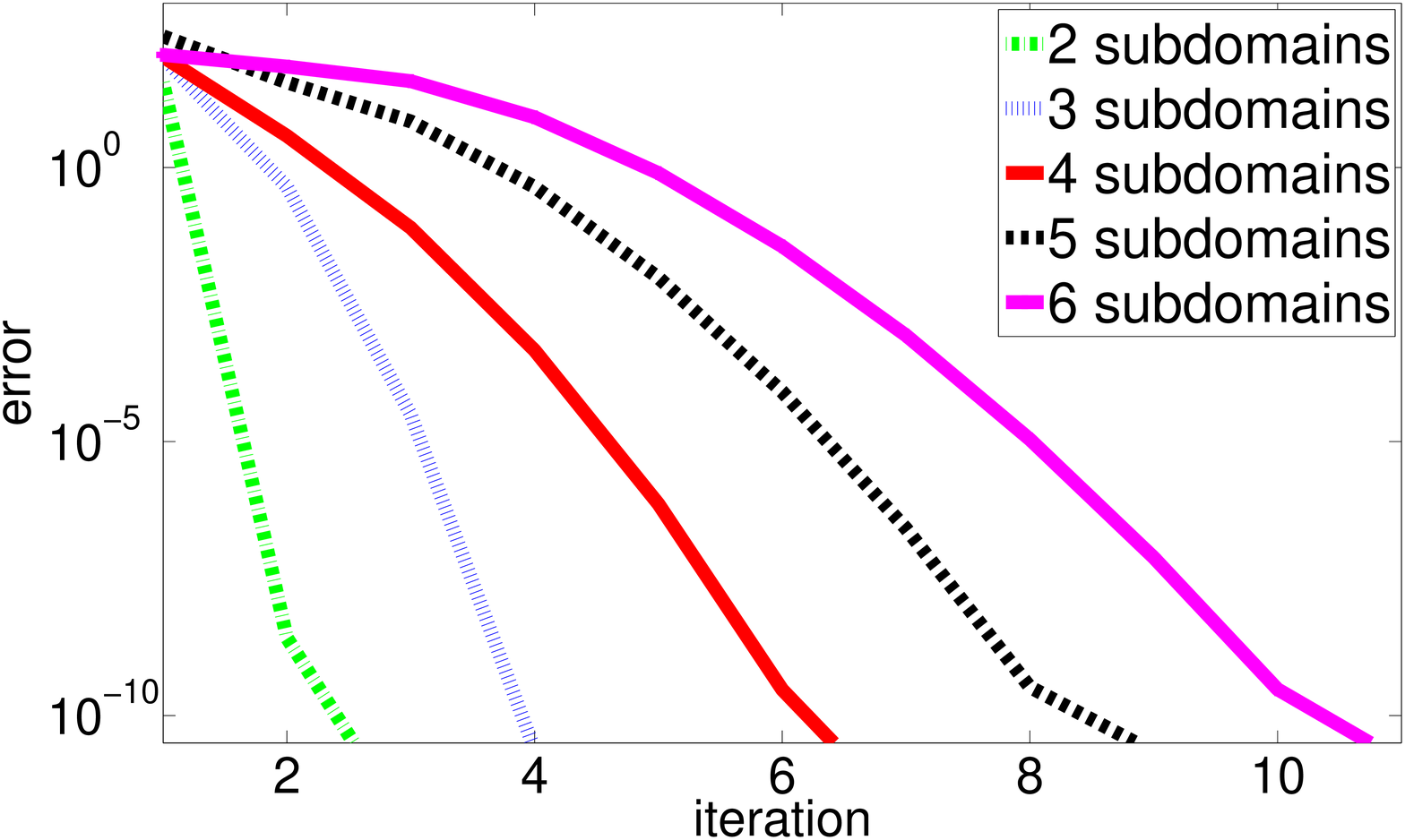}
  \caption{Convergence of NNWR with four subdomains and various
    relaxation parameters on the left, and dependence of NNWR on the
    number of subdomains for $\theta=1/4$ on the right}
  \label{NumFig2}
\end{figure}

\begin{figure}
  \centering
  \includegraphics[width=0.49\textwidth]{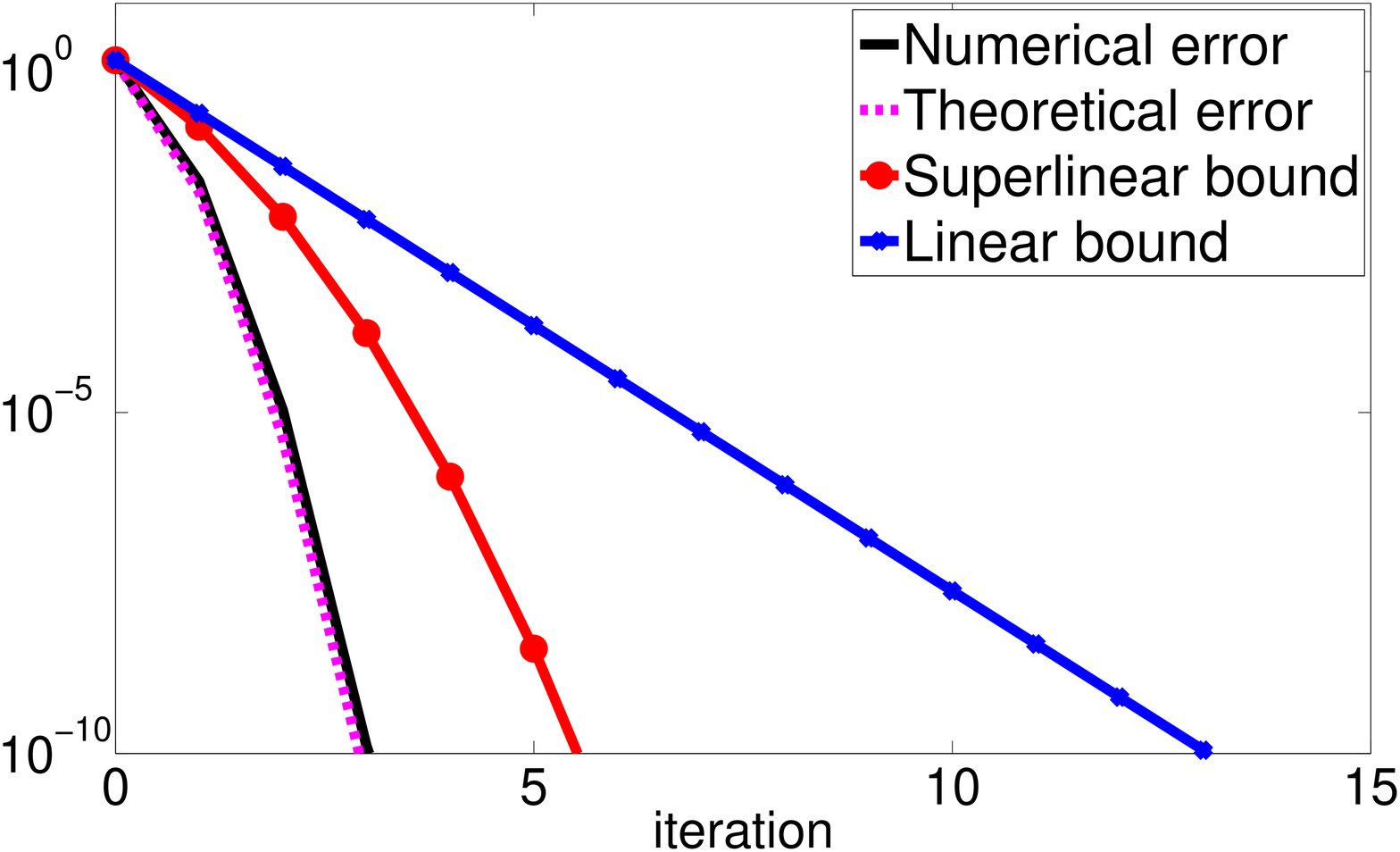}
  \includegraphics[width=0.49\textwidth]{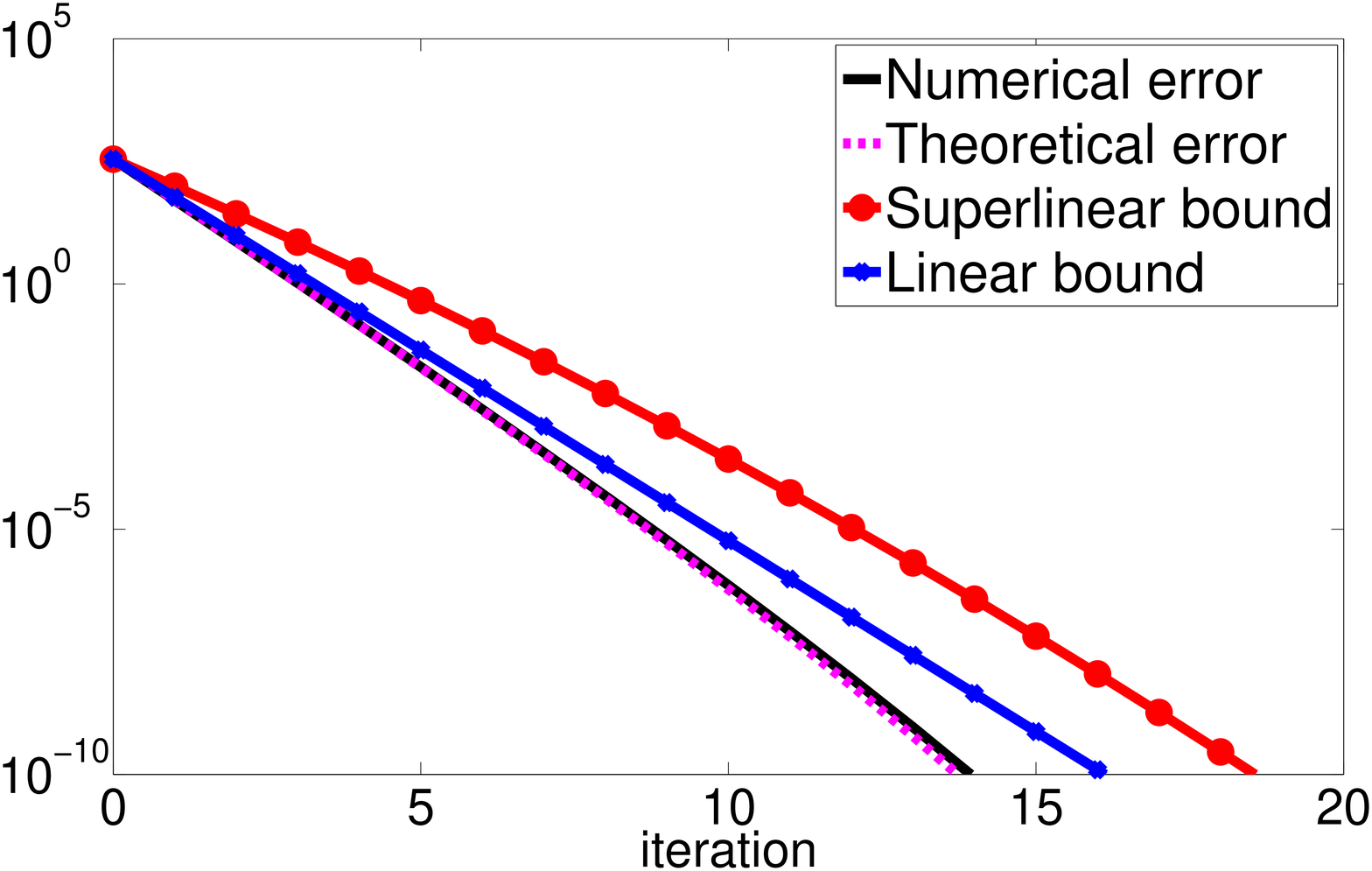}
  \caption{Comparison of the numerically measured convergence rates
    and the theoretical error estimates for DNWR for $\kappa(x)=1$ with $T=2$ on the
    left, and $T=50$ on the right}
  \label{NumFig3}
\end{figure}
We now compare the numerical behavior of DNWR and NNWR with our
theoretical estimates in Sections \ref{Section2} and \ref{Section3}.
In Figure \ref{NumFig3}, we show for the DNWR algorithm a comparison
between the numerically measured convergence for the discretized
problem, the theoretical convergence for the continuous model problem
(calculated using inverse Laplace transforms), and the linear and
superlinear convergence estimates shown in Theorem \ref{Theorem2}, for
$a=3$, $b=2$, $\kappa(x)=1$. We see that for a short time interval,
$T=2$, the algorithm converges superlinearly, and the superlinear
bound is quite accurate. For the long time interval $T=50$, the
algorithm converges linearly, and the linear convergence estimate is
now more accurate.
\begin{figure}
  \centering
  \includegraphics[width=0.49\textwidth]{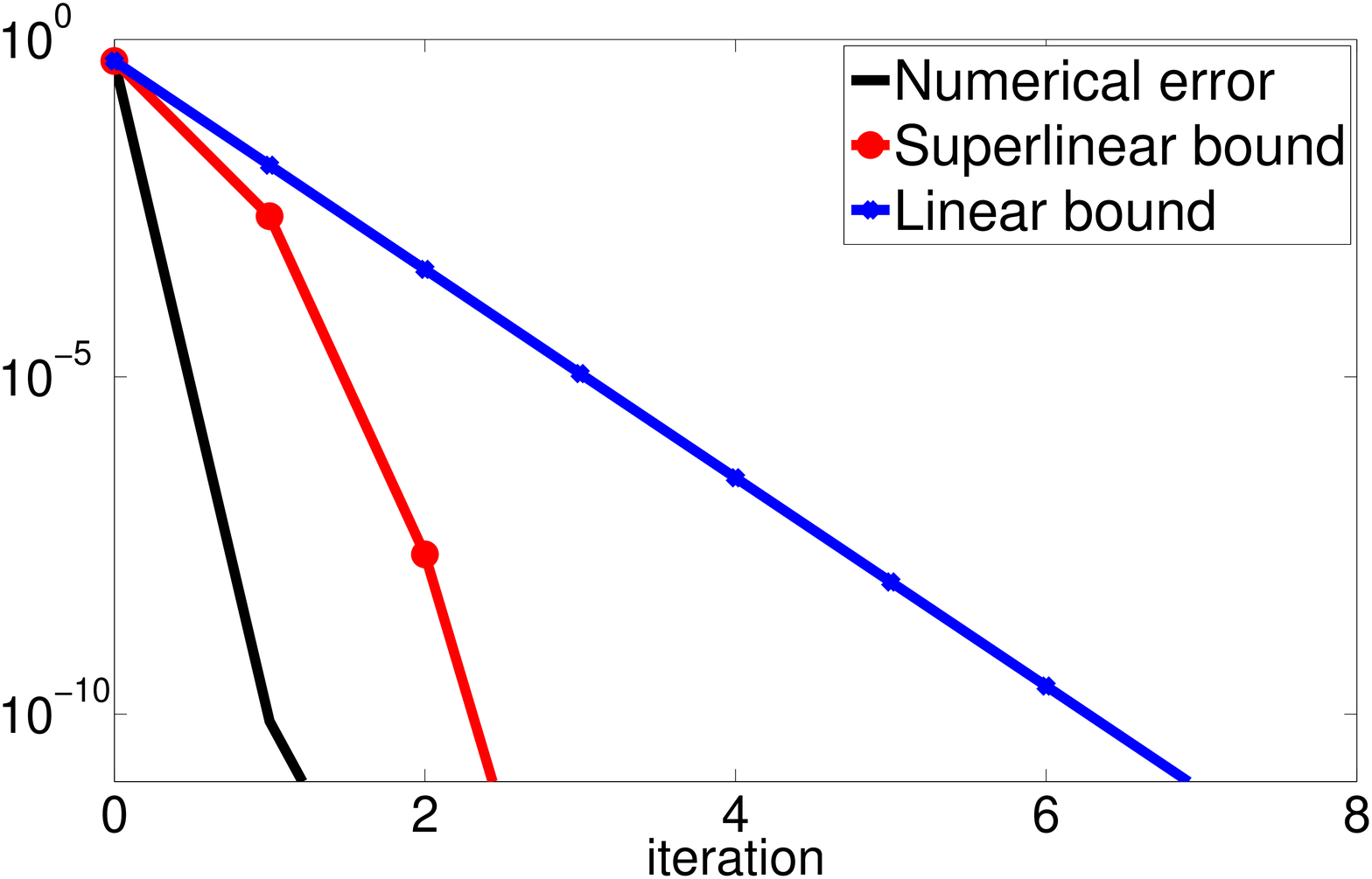}
  \includegraphics[width=0.49\textwidth]{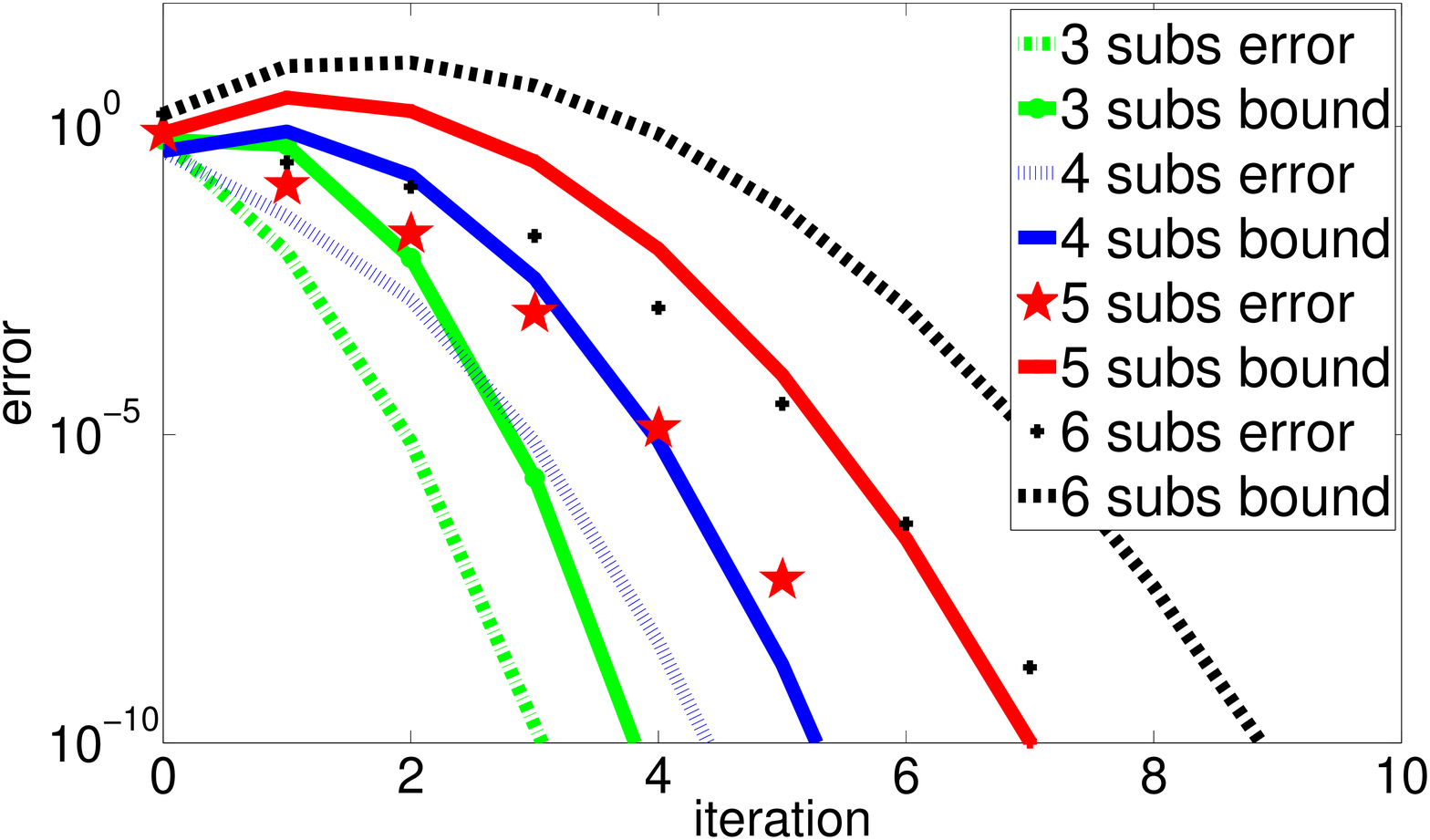}
  \caption{Comparison of the numerically measured convergence rates
    and the theoretical error estimates for NNWR for $\kappa(x)=1$ with
    $\theta=1/4$ and $T=2$, on the left for two subdomains, and on the
    right for many subdomains}
  \label{NumFig4}
\end{figure}
Similarly, we show in Figure \ref{NumFig4} a comparison of the
numerically measured convergence for the NNWR algorithm for
$\theta=1/4$ and $\kappa(x)=1$, and the theoretical estimates from Theorem
\ref{Theorem4}. On the left, we show the results for the two subdomain
case (subdomain lengths are as in the first line of Table \ref{Table1}), where we also plotted the linear estimate from \cite{Kwok}, and on
the right, we show the results for the case of many subdomains of equal length for $\Omega=(0,6)$.

We now compare in Figure \ref{NumFig5} 
\begin{figure}
  \centering
  \includegraphics[width=0.49\textwidth]{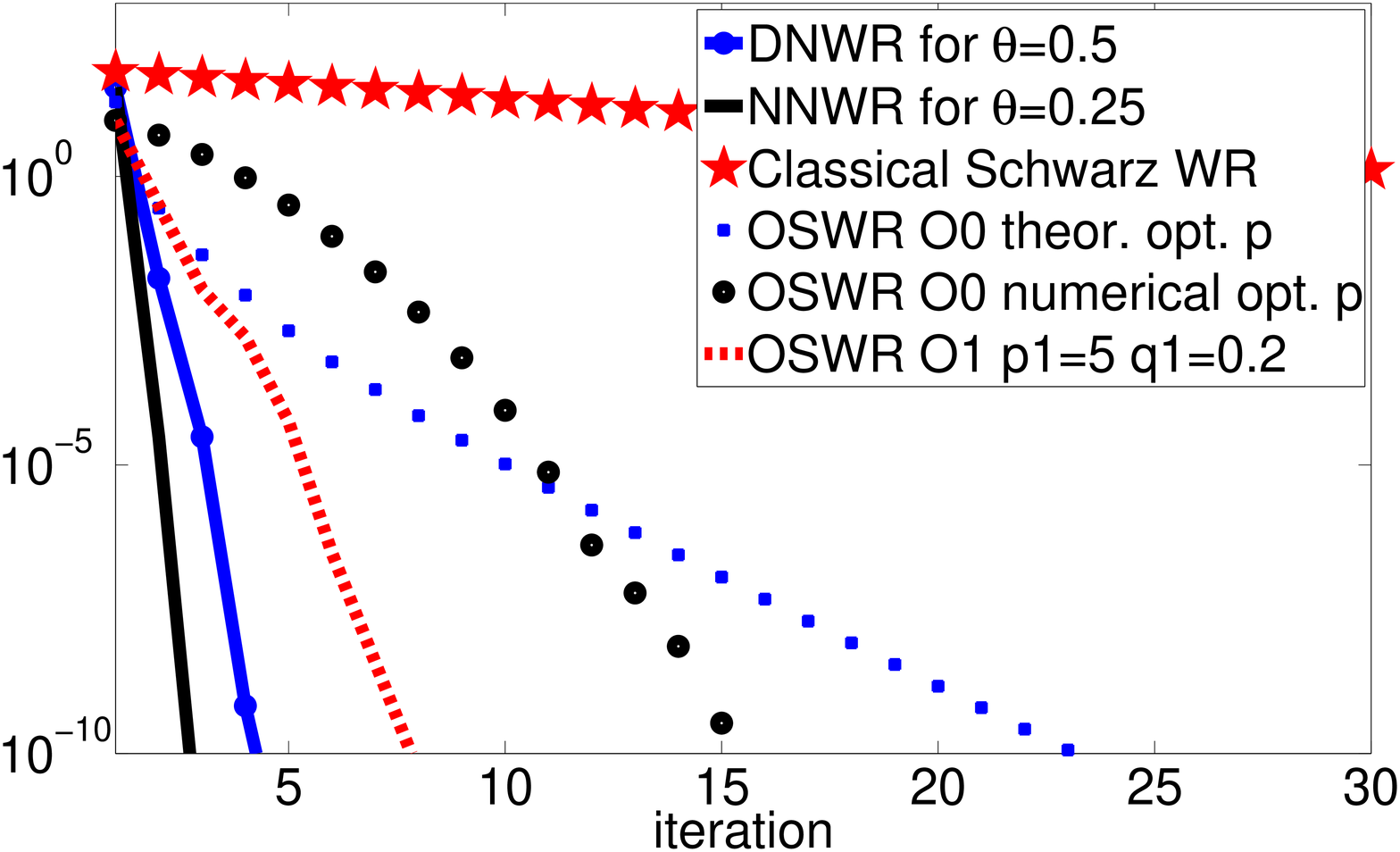}
  \caption{Comparison of DNWR and NNWR with Schwarz waveform
    relaxation}
  \label{NumFig5}
\end{figure}
the performance of the DNWR and NNWR algorithms for two subdomains
with the Schwarz Waveform Relaxation algorithms from \cite{GH1,BGH}
with overlap. We use an overlap of length $2\Delta x$, where $\Delta
x=1/50$. We observe that the DNWR and NNWR algorithms converge faster
than the overlapping Schwarz WR iteration. Only a higher order
optimized Schwarz waveform relaxation algorithm comes close to the
performance of the DNWR algorithm in this experiment.

We show an experiment for the NNWR algorithm in two dimension for the following model problem
$$ \partial_t u-\left(\partial_{xx}u+\partial_{yy}u\right)=0,u(x,y,0)=\sin(2\pi x)\sin(3\pi y).$$
We decompose our domain $\Omega:=(0,1)\times (0,\pi)$ into three non-overlapping subdomains
$\Omega_1=(0,2/5)\times (0,\pi)$, $\Omega_2=(2/5,3/4)\times (0,\pi)$,
$\Omega_3=(3/4,1)\times (0,\pi)$, see Figure \ref{NumFig6} on the
left. On the right, we plot the numerical errors of the NNWR algorithm for various $\theta$ and the theoretical estimates from Theorem \ref{2DTheorem} for $\theta= 1/4$ and again observe superlinear convergence.

\begin{figure}
  %\centering
  \includegraphics[width=0.47\textwidth]{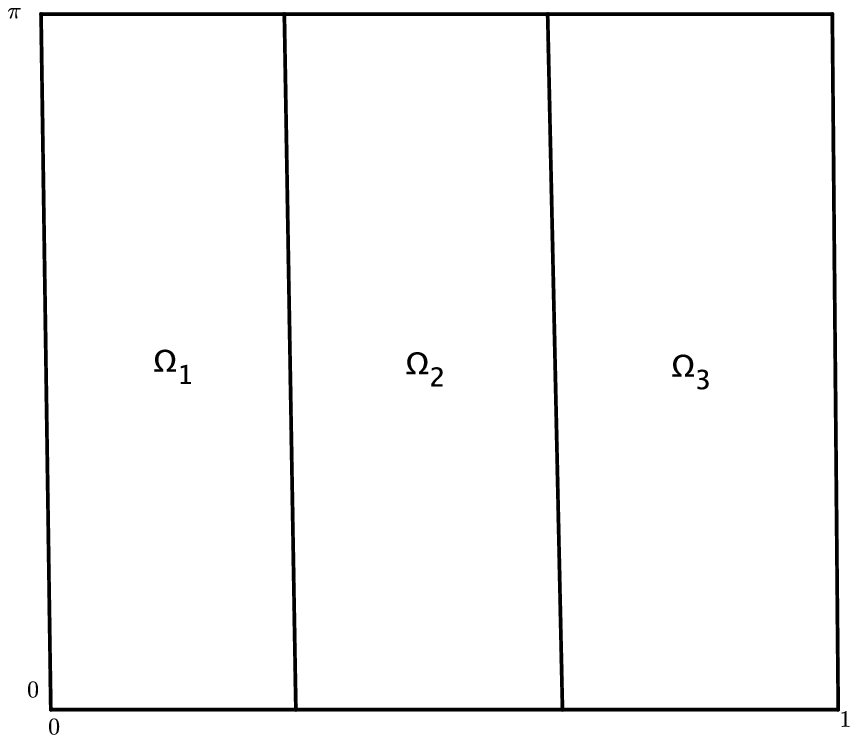}
  \includegraphics[width=0.5\textwidth]{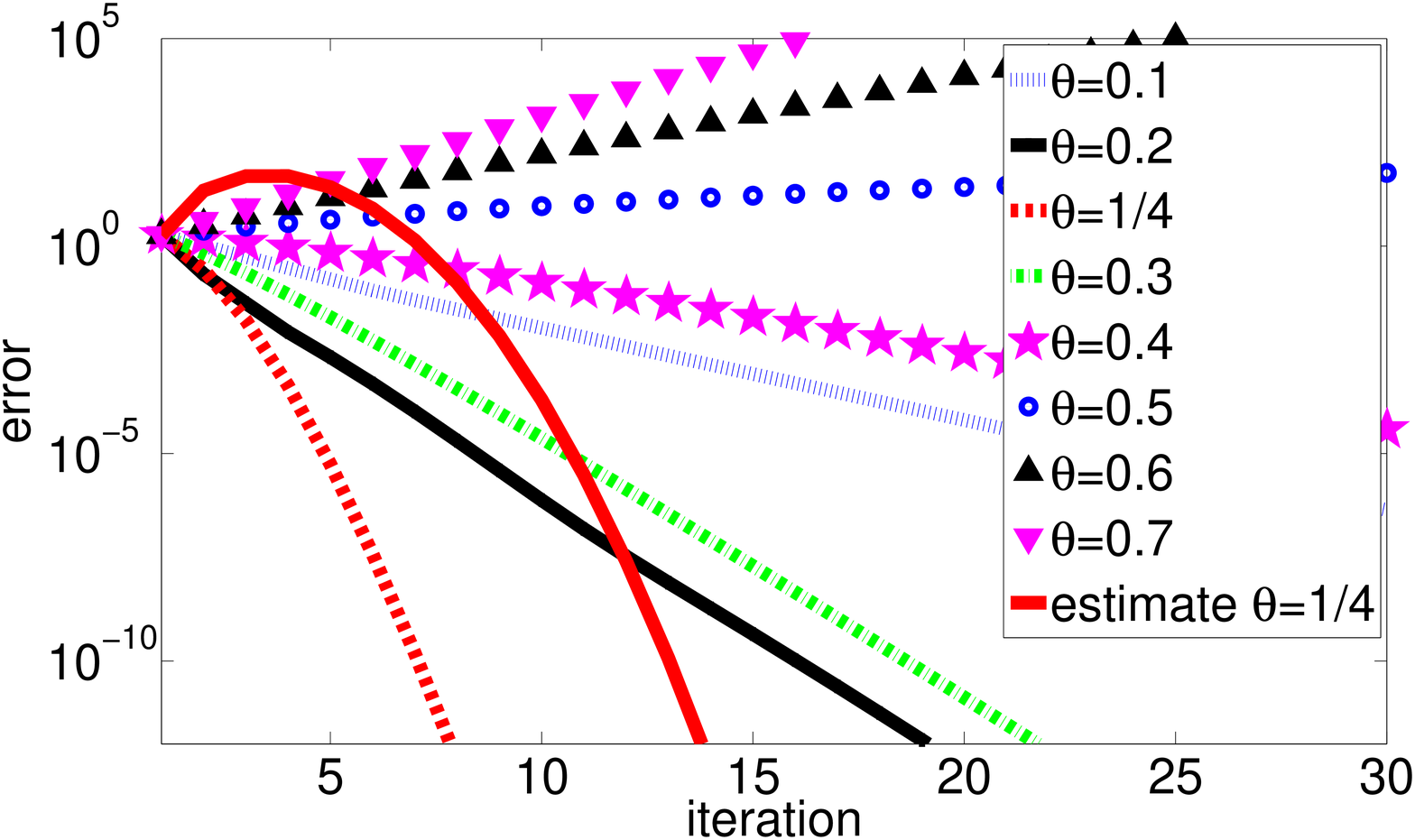}
  \caption{Decomposition of 2d domain into strips on the left, and convergence of NNWR using various relaxation parameters $\theta$ for $T=0.2$ on the right}
  \label{NumFig6}
  \end{figure}

We conclude this section with a numerical experiment not covered by
our analysis: we decompose the two dimensional domain
$\Omega:=(0,1)\times (0,1)$ into four non-overlapping subdomains
$\Omega_1=(0,1/2)\times (0,1/2)$, $\Omega_2=(0,1/2)\times (1/2,1)$,
$\Omega_3=(1/2,1)\times (1/2,1)$, $\Omega_4=(1/2,1)\times (0,1/2)$,
such that a cross point is present, see Figure \ref{NumFig7} on the
left. On the right, we show that the convergence of the NNWR algorithm
is again superlinear.

\begin{figure}
  %\centering
  \includegraphics[width=0.47\textwidth]{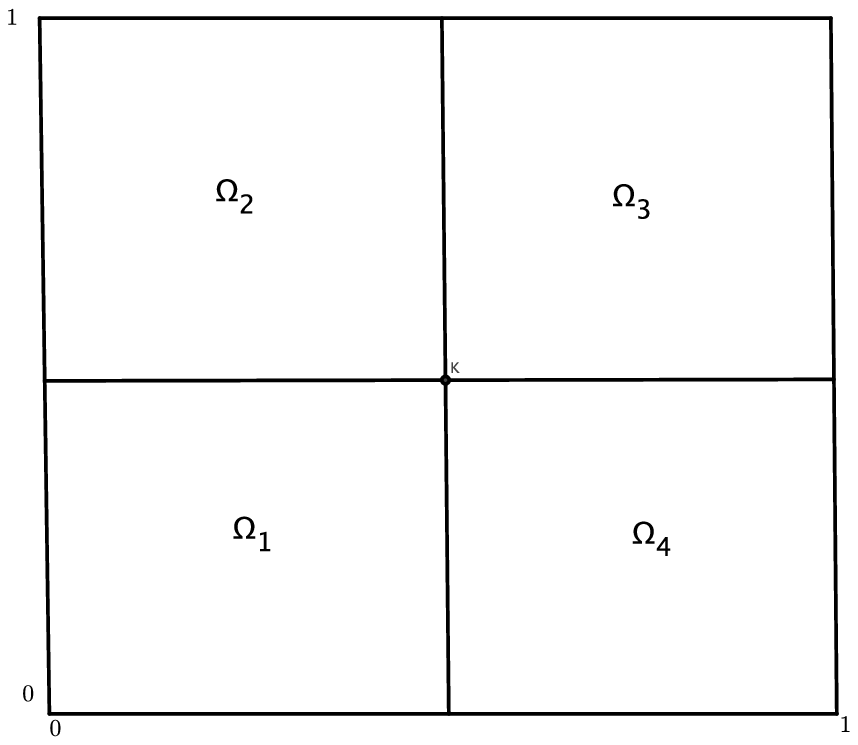}
  \includegraphics[width=0.5\textwidth]{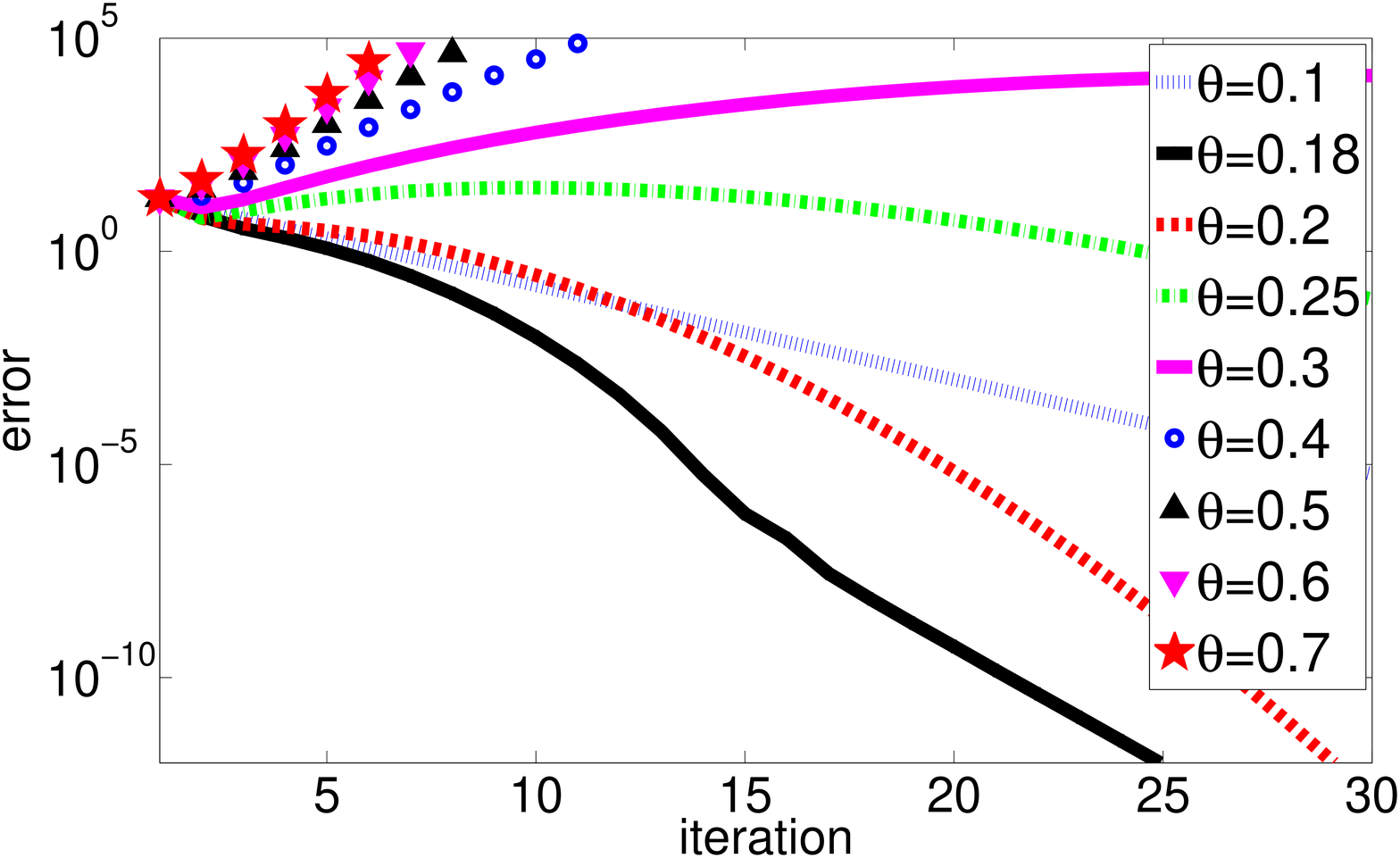}
  \caption{Decomposition of 2d domain with a crosspoint on the left, and convergence of NNWR using various relaxation parameters $\theta$ for $T=0.2$ on the right}
  \label{NumFig7}
  \end{figure}
%\afterpage{\clearpage}

\section{Conclusions}

We introduced two new classes of space-time parallel algorithms, the
Dirichlet-Neumann waveform relaxation (DNWR) and the Neumann-Neumann
waveform relaxation (NNWR) algorithms. For  the
one-dimensional heat equation, we proved superlinear convergence for
both algorithms for a particular choice of the relaxation
parameter. For the NNWR, our convergence estimate holds for a
decomposition into many subdomains, and we also gave an extension to
two spatial dimensions. We are currently working on the generalization
of our analysis for the DNWR algorithm to the case of many subdomains
and higher spatial dimensions, and we are also studying the
convergence for $\theta\ne 1/2$ in more detail.

\bibliographystyle{siam}
\bibliography{paperfinal} 
\nocite{*}

\end{document}